\pdfoutput=1
\RequirePackage{ifpdf}
\ifpdf % We are running pdfTeX in pdf mode
\documentclass[pdftex]{sigma}
\else
\documentclass{sigma}
\fi

\numberwithin{equation}{section}

{\theoremstyle{definition}
\newtheorem{dfn}{Definition}[section]

\newtheorem{rmk}[dfn]{Remark}}

\newtheorem{thm}[dfn]{Theorem}
\newtheorem{prp}[dfn]{Proposition}
\newtheorem{lem}[dfn]{Lemma}
\newtheorem{cor}[dfn]{Corollary}
\newtheorem{con}[dfn]{Conjecture}

\newcommand{\bbC}{\mathbb{C}}
\newcommand{\bbZ}{\mathbb{Z}}
\def\vector#1{\mbox{\boldmath $#1$}}

\begin{document}
\allowdisplaybreaks

\newcommand{\arXivNumber}{1801.09939}

\renewcommand{\thefootnote}{}

\renewcommand{\PaperNumber}{101}

\FirstPageHeading

\ShortArticleName{Macdonald Polynomials of Type $C_n$ with One-Column Diagrams}

\ArticleName{Macdonald Polynomials of Type $\boldsymbol{C_n}$ with One-Column\\ Diagrams and Deformed Catalan
Numbers\footnote{This paper is a~contribution to the Special Issue on Elliptic Hypergeometric Functions and Their Applications. The full collection is available at \href{https://www.emis.de/journals/SIGMA/EHF2017.html}{https://www.emis.de/journals/SIGMA/EHF2017.html}}}

\Author{Ayumu HOSHINO~$^\dag$ and Jun'ichi SHIRAISHI~$^\ddag$}

\AuthorNameForHeading{A.~Hoshino and J.~Shiraishi}

\Address{$^\dag$~Hiroshima Institute of Technology, 2-1-1 Miyake, Hiroshima 731-5193, Japan}
\EmailD{\href{mailto:a.hoshino.c3@it-hiroshima.ac.jp}{a.hoshino.c3@it-hiroshima.ac.jp}}

\Address{$^\ddag$~Graduate School of Mathematical Sciences, University of Tokyo,\\
\hphantom{$^\ddag$}~Komaba, Tokyo 153-8914, Japan}
\EmailD{\href{mailto:shiraish@ms.u-tokyo.ac.jp}{shiraish@ms.u-tokyo.ac.jp}}

\ArticleDates{Received January 31, 2018, in final form September 11, 2018; Published online September 20, 2018}

\Abstract{We present an explicit formula for the transition matrix $\mathcal{C}$ from the type~$C_n$ degeneration of the Koornwinder polynomials
$P_{(1^r)}(x\,|\,a,-a,c,-c\,|\,q,t)$ with one column diagrams, to the type $C_n$ monomial symmetric polynomials $m_{(1^{r})}(x)$. The entries of the matrix $\mathcal{C}$ enjoy a~set of three term recursion relations, which can be regarded as a~$(a,c,t)$-deformation of the one for the Catalan triangle or ballot numbers. Some transition matrices are studied associated with the type $(C_n,C_n)$ Macdonald polynomials $P^{(C_n,C_n)}_{(1^r)}(x\,|\,b;q,t)= P_{(1^r)}\big(x\,|\,b^{1/2},-b^{1/2},q^{1/2}b^{1/2},-q^{1/2}b^{1/2}\,|\,q,t\big)$. It is also shown that the $q$-ballot numbers appear as the Kostka polynomials, namely in the transition matrix from the Schur polynomials $P^{(C_n,C_n)}_{(1^r)}(x\,|\,q;q,q)$ to the Hall--Littlewood polynomials $P^{(C_n,C_n)}_{(1^r)}(x\,|\,t;0,t)$.}


\Classification{33D52; 33D45}

\renewcommand{\thefootnote}{\arabic{footnote}}
\setcounter{footnote}{0}

\section{Introduction}
The aim of this article is to investigate the transition matrix $\mathcal{C}$, which describes the expansion of the type~$C_n$ degeneration of the Koornwinder polynomials~\cite{K} $P_{(1^r)}(x\,|\,a,-a,c,-c\,|\,q,t)$ with one column diagrams, in terms of the type~$C_n$ monomial symmetric polynomials~$m_{(1^r)}(x)$. As for our convention of notation, see Section~\ref{Koornwinder}. On this course, we found that certain deformations appear, associated with the Catalan triangle or ballot numbers, and binomial coefficients. We refer the readers to~\cite{S} concerning the Catalan triangle numbers, and~\cite{A, FH} for the $q$-Catalan and $q$-ballot numbers. For simplicity, write $P^{(C_n)}_{(1^r)}=P_{(1^r)}(x\,|\,a,-a,c,-c\,|\,q,t)$.

\begin{thm}\label{MAIN} Let $n\in \mathbb{Z}_{>0}$. Let ${\bf P}^{(n)}$ and ${\bf m}^{(n)}$ be the infinite column vectors
\begin{gather*}
{\bf P}^{(n)}={}^t\big(P^{(C_n)}_{(1^{n})},\ldots,P^{(C_n)}_{(1)},P^{(C_n)}_{\varnothing},0,0,0,\ldots \big), \\
{\bf m}^{(n)}={}^t\big(m_{(1^{n})},\ldots,m_{(1)},m_{\varnothing},0,0,0,\ldots \big).
\end{gather*}
There exist a unique infinite transition matrix $\mathcal{C} =(\mathcal{C} )_{i,j\in \mathbb{Z}_{\geq 0}}$ satisfying the conditions
\begin{subequations}
\begin{gather}
\mathcal{C} \text{ is upper triangular, namely } i>j \text{ implies } \mathcal{C}_{ij}=0,\\
\mathcal{C} \text{ is even, namely } i+j \text{ is odd implies } \mathcal{C}_{ij}=0,\\
\mathcal{C}_{ij} \text{ are rational functions in $a$, $c$ and~$t$ which do not depend on~$n$} \nonumber\\
\mbox{and we have } {\bf P}^{(n)} = \mathcal{C} {\bf m}^{(n)} \text{ for all } n\geq 1 \ \text{$($stability$)$}. \label{STAB}
\end{gather}
\end{subequations}
This transition matrix $\mathcal{C}$ is uniquely characterized by the $(a,c,t)$-deformed Catalan triangle type three term recursion relations
\begin{subequations}
\begin{gather}
\mathcal{C}_{0,0}=1, \qquad \mathcal{C}_{i-1,i-1}=\mathcal{C}_{i,i}, \qquad i=1,2,3,\ldots, \label{Catalan-1}\\
f(t)\mathcal{C}_{1,j-1}=\mathcal{C}_{0,j}, \qquad j=2,4,6,\ldots,\label{Catalan-2}\\
\mathcal{C}_{i-1,j-1}+f\big( t^{i+1} \big)\mathcal{C}_{i+1,j-1}=\mathcal{C}_{i,j}, \qquad i+j \ {\rm even}, \ 0<i< j,\label{Catalan-3}
\end{gather}
\end{subequations}
where we have used the notation
\begin{gather}
f(s)=\dfrac{(1-1/s )\big(1-t^{2}/s a^2c^2 \big)\big(1+t/s a^2\big)\big(1+t/sc^2\big)} {\big(1-t/s^{2}a^2c^2\big)\big(1-t^{3}/s^{2}a^2c^2\big)}. \label{f[s]}
\end{gather}
\end{thm}

A proof of this is presented in Section~\ref{ProofMAIN}. The solution to the three term recursion relations (\ref{Catalan-1}), (\ref{Catalan-2}) and (\ref{Catalan-3}) for $\mathcal{C}_{i,j}$ given in terms of the function $f(s)$ is presented in Proposition~\ref{FFF}.

Consider the Macdonald polynomials of types $(C_n,C_n)$ and $(D_n,D_n)$ \cite{M2, RW, St}
\begin{gather*}
P^{(C_n,C_n)}_{(1^r)}(x\,|\,b;q,t)= P_{(1^r)}\big(x\,|\,b^{1/2},-b^{1/2},q^{1/2}b^{1/2},-q^{1/2}b^{1/2}\,|\,q,t\big), \\
P^{(D_n,D_n)}_{(1^r)}(x\,|\,q,t)=P_{(1^r)}\big(x\,|\,1,-1,q^{1/2},-q^{1/2}\,|\,q,t\big).
\end{gather*}

\begin{cor} When $b=t=q$, the Macdonald polynomials of type $(C_n,C_n)$ become the Schur polynomials $s_{\lambda}(x)=s^{(C_n)}_{\lambda}(x)$ of type $C_n$. In this case we have $f\big(t^{i+1}\big)=1$ for $i\geq 0$, indicating that the recursion relations \eqref{Catalan-1}--\eqref{Catalan-3}
reduces to the ones for the ordinary Catalan triangle $($or ballot$)$ numbers. Therefore it holds that
\begin{gather}
s^{(C_n)}_{(1^r)}(x)=P^{(C_n,C_n)}_{(1^r)}(x\,|\,q;q,q) =
\sum_{k=0}^{\lfloor{r \over 2}\rfloor}{n-r+1\over n-r+k+1}{n-r+2k \atopwithdelims() k} m_{(1^{r-2k})}(x), \label{s-m}
\end{gather}
where $\binom{m}{j}={m(m-1)\cdots(m-j+1) \over j!}$ denotes the ordinary binomial coefficient.
\end{cor}

\begin{cor} When $b=1$ and $t=q$, the Macdonald polynomials of type $(C_n,C_n)$ become the Schur polynomials $s_{\lambda}(x)=s^{(D_n)}_{\lambda}(x)$ of type $D_n$. $($See Remark~{\rm \ref{limit-D}} below.$)$ In this case we have $f(t)=2$ and $f\big(t^{i+1}\big)=1$ for $i> 0$, and the recursion relations \eqref{Catalan-1}--\eqref{Catalan-3} reduces to the ones for $($the half of$)$ the ordinary Pascal triangle. We have
\begin{gather}
s^{(D_n)}_{(1^r)}(x)=P^{(D_n,D_n)}_{(1^r)}(x\,|\,q,q) =
\sum_{k=0}^{\lfloor{r \over 2}\rfloor}{n-r+2k \atopwithdelims() k} m_{(1^{r-2k})}(x) .\label{sD-m}
\end{gather}
\end{cor}

\begin{rmk}\label{limit-D} To be precise, when $\ell(\lambda)=n$, the polynomial $P^{(C_n,C_n)}_{\lambda}(x\,|\,1;q,t)$ (or $m_{\lambda}$) has to
be further decomposed in terms of the type $D_n$ Macdonald (or monomial) polynomials~\cite{RW, St}, since the Weyl group is smaller than the one for $C_n$. Such a decomposition is easy but takes some space for a separate treatment. Therefore throughout in this paper, we do not go into the actual details, leaving this to the interested reader.
\end{rmk}

The first few terms of (\ref{s-m}) and (\ref{sD-m}) read
\begin{gather*}
 \left(
\begin{matrix}
s_{(1^{n})}^{(C_n)}\vspace{1mm}\\
s_{(1^{n-1})}^{(C_n)}\vspace{1mm}\\
s_{(1^{n-2})}^{(C_n)}\vspace{1mm}\\
s_{(1^{n-3})}^{(C_n)}\vspace{1mm}\\
\vdots
\end{matrix}
\right)=
\left(
\begin{array}{@{}ccccccccccc@{}}
1& &1 & & 2 & & 5 & & 14 &&\cdots\\
&1& & 2 & & 5 & &14 && 42& \\
&&1& & 3 & & 9 & &28 && \cdots\\
&&&1& & 4 & & 14 & &48 & \\
&&&&\ddots& &\ddots& & \ddots
\end{array}
 \right)
 \left(
\begin{matrix}
m_{(1^{n})}\\
m_{(1^{n-1})}\\
m_{(1^{n-2})}\\
m_{(1^{n-3})}\\
\vdots
\end{matrix}
\right), \\
\left(
\begin{matrix}
s_{(1^{n})}^{(D_n)}\vspace{1mm}\\
s_{(1^{n-1})}^{(D_n)}\vspace{1mm}\\
s_{(1^{n-2})}^{(D_n)}\vspace{1mm}\\
s_{(1^{n-3})}^{(D_n)}\\
\vdots
\end{matrix}
\right)=
\left(
\begin{array}{@{}ccccccccccc@{}}
1& &2 & & 6 & & 20 & & 70 &&\cdots\\
&1& & 3 & & 10 & &35 && 126& \\
&&1& & 4 & & 15 & &56 && \cdots\\
&&&1& & 5 & & 21 & &84 & \\
&&&&\ddots& &\ddots& & \ddots
\end{array}
 \right)
 \left(
\begin{matrix}
m_{(1^{n})}\\
m_{(1^{n-1})}\\
m_{(1^{n-2})}\\
m_{(1^{n-3})}\\
\vdots
\end{matrix}
\right).
\end{gather*}

As an application of our results obtained in this paper, we calculate the transition matrix from the Schur polynomials to the Hall--Littlewood polynomials, namely the Kostka polynomials, associated with one column diagrams.
\begin{dfn}
Let $K^{(C_n)}_{(1^r)(1^{r-2j})}(t)$ and $K^{(D_n)}_{(1^r)(1^{r-2j})}(t)$ be the transition coefficients defined by
%\begin{subequations}
\begin{gather*}
s^{(C_n)}_{(1^r)}(x)=\sum_{j=0}^{\lfloor {r\over 2}\rfloor}K^{(C_n)}_{(1^r)(1^{r-2j})}(t)P^{(C_n,C_n)}_{(1^{r-2j})}(x\,|\,t;0,t),\\
s^{(D_n)}_{(1^r)}(x)=\sum_{j=0}^{\lfloor {r\over 2}\rfloor}K^{(D_n)}_{(1^r)(1^{r-2j})}(t)P^{(D_n,D_n)}_{(1^{r-2j})}(x\,|\,0,t).
\end{gather*}
%\end{subequations}
\end{dfn}

\begin{thm}\label{KOSTKAthm} The $K^{(C_n)}_{(1^r)(1^{r-2j})}(t)$ and $K^{(D_n)}_{(1^r)(1^{r-2j})}(t)$ are polynomials in $t$ with nonnegative integral coefficients. We have
%\begin{subequations}
\begin{align*}
K^{(C_n)}_{(1^r)(1^{r-2j})}(t)
 & =
 t^{2j}{[n-r+1]_{t^2}\over [n-r+j+1]_{t^2}}
{n-r+2j \atopwithdelims[] j}_{t^2}
=
 {n-r+2j \atopwithdelims[] j}_{t^2} - {n-r+2j \atopwithdelims[] j-1}_{t^2},\nonumber\\
K^{(D_n)}_{(1^r)(1^{r-2j})}(t)
&=
t^j {1+t^{n-r}\over 1+t^{n-r+2j}}
\left[ n-r+2j\atop j\right]_{t^2} \\
&=
t^{n-r+j}
\left[ n-r+2j-1\atop j-1\right]_{t^2}+
t^j
\left[ n-r+2j-1\atop j\right]_{t^2}. \nonumber
\end{align*}
%\end{subequations}
Here we have used the notation for the $q$-integer $[n]_q$, the $q$-factorial $ [n]_q!$ and the $q$-binomial coefficient $ \left[m\atop j\right]_q$ as
\begin{gather*}
 [n]_q={1-q^n \over 1-q}, \qquad [n]_q!=[1]_q[2]_q\cdots [n]_q,
\qquad \left[m\atop j\right]_q=\prod_{k=1}^j {[m-k+1]_q\over [k]_q}={[m]_q!\over [j]_q![m-j]_q!}.
\end{gather*}
\end{thm}

As for our proof of this, see Section~\ref{Kostka}.

\begin{rmk}Note that the $K^{(C_n)}(t)$'s are essentially given by the $t^2$-deformed ballot num\-bers~\cite{A} (the case $n=r$ corresponds to the $t^2$-deformation of the Catalan numbers~\cite{FH}), and the $K^{(C_n)}(t)$'s by a version of the $t$-deformed binomial numbers.
\end{rmk}

First few entries of $K^{(C_n)}(t)$ read
\begin{gather*}
\left(
\begin{array}{@{}ccccccccccc@{}}
1& &t^2 & & t^4+t^8 & & t^6+t^{10}+t^{12}+t^{14}+t^{18} \\
&1& & t^2+t^4 & & \displaystyle{ t^4+t^{6}+t^{8}\atop \qquad +t^{10}+t^{12} }& &\cdots \\
&&1& & t^2+t^4+t^6 & &\displaystyle { t^4+t^6+2t^8+t^{10}\atop \qquad \qquad +2 t^{12}+t^{14}+t^{16}} \\
&&&1& & t^2+t^4+t^6+t^8 & &\cdots \\
&&&&\ddots& & \ddots
\end{array}
 \right),
\end{gather*}
and for $K^{(D_n)}(t)$ we have
\begin{gather*}
\left(
\begin{array}{@{}ccccccccccc@{}}
1& &2 t & & 2 t^2+2t^4+2t^6 & &\cdots\\
&1& & t+t^2+t^3 & & \displaystyle{t^{2}+t^{3}+t^{4}+t^{5}+2t^{6}\atop\qquad +t^{7}+t^{8}+t^{9}+t^{10}\qquad } & \\
&&1& & t+2 t^3+t^5 & &\cdots\\
&&&1& & t+ t^3+t^4+t^5+t^7 & & \\
&&&&\ddots& & \ddots
\end{array}
 \right).
\end{gather*}

The present article is organized as follows. In Section~\ref{Collection}, several transition formulas obtained in this paper are summarized for the convenience of reading. Then we present a proof of our main result Theorem~\ref{MAIN}. In Sections~\ref{Koornwinder} and~\ref{one-column}, we use Mimachi's kernel function identity to have a description of the Koornwinder polynomials and the Macdonald polynomials of type $(C_n,C_n)$ with one column diagrams. Sections~\ref{MatrixInverse} and~\ref{BBt} form the core of the technical part of this article. In Section~\ref{MatrixInverse}, Bressoud's matrix inversion is applied to invert the formula for the Macdonald polynomials of type $(C_n,C_n)$ with one column diagrams. In Section~\ref{BBt}, the four term relations for $B(s,j)$ and $\widetilde{B}(s,j)$ are derived. In Section~\ref{SEC-C} is given the basic properties for the transition matrix $\mathcal{C}$. In Section~\ref{DEGEN}, we study some degenerate cases, including the calculation of the Kostka polynomials. Some conjectures are presented in Section~\ref{CON}, concerning the asymptotically free type eigenfunctions for type $C_n$ when $b=t$. It is quite conceivable that Theorem~\ref{MAIN} admits an elliptic generalization in terms
of the $BC_n$ abelian functions~\cite{R1, R2}, but we have not yet attempted to formulate such a generalization.

Throughout the paper, we use the standard notation (see~\cite{GR})
\begin{gather*}
(z;q)_\infty =\prod_{k=0}^{\infty}\big(1-q^k z\big), \qquad (z;q)_k=\frac{(z;q)_{\infty}}{(q^kz;q)_{\infty}},\qquad k\in\mathbb{Z}, \nonumber \\
(a_1, a_2, \ldots, a_r;q)_k = (a_1;q)_k (a_2;q)_k \cdots (a_r;q)_k, \qquad k\in\mathbb{Z}, \nonumber \\
{}_{r+1}\phi_r\left[ {a_1,a_2,\ldots,a_{r+1}\atop b_1,\dots,b_{r}};q,z\right]= \sum_{n=0}^\infty {(a_1, a_2, \ldots, a_{r+1};q)_n \over
(q, b_1, b_2, \ldots, b_{r};q)_n} z^n ,\nonumber\\
{}_{r+1}W_r(a_1;a_4,a_5,\ldots,a_{r+1};q,z)= {}_{r+1}\phi_r\left[ {a_1,q a_1^{1/2},-q a_1^{1/2},a_4,\ldots,a_{r+1}\atop
a_1^{1/2},-a_1^{1/2},q a_1/a_4,\dots,qa_1/a_{r+1}};q,z\right].\nonumber
\end{gather*}

\section{Collection of transition formulas and proof of Theorem~\ref{MAIN}}\label{Collection}

In this section, we collect several transformation formulas which we need to establish Theorem~\ref{MAIN}, giving brief explanations about our ideas and methods for their derivations.

\subsection[Koornwinder polynomials $P_{(1^r)}(x\,|\,a,b,c,d\,|\,q,t)$ with one column diagrams]{Koornwinder polynomials $\boldsymbol{P_{(1^r)}(x\,|\,a,b,c,d\,|\,q,t)}$\\ with one column diagrams}

In \cite{FHNSS}, we studied some explicit formulas for the Koornwinder polynomials \cite{K} with one-row diagrams. The results were interpreted as certain summations over the sets of tableaux of types~$C_n$ and~$D_n$. While using the same technique as in~\cite{FHNSS}, but replacing the Cauchy type kernel function by Mimachi's dual-Cauchy type one (as to the kernel functions, see~\cite{KNS,Mi}), we can study an explicit formula for the Koornwinder polynomials with one column diagrams. Mimachi's kernel function~\cite{Mi} intertwines the action of the Koornwinder operator of type $BC_n$ to the one for $BC_1$ (namely for the Askey--Wilson operator) which in turn acts on the Askey--Wilson eigenfunction. To perform explicit calculations based on this idea, as in the one-row diagram case, we need the fourfold summation formula for the Askey--Wilson eigenfunction~\cite{HNS}. The details will be given in Sections~\ref{Koornwinder} and~\ref{one-column}.

Specializing the parameters of the Koornwinder polynomials, we obtain the Macdonald polynomials of types $C_n$ and $D_n$ with one column diagram. In these particular limits, the fourfold summation (for the Askey--Wilson eigenfunction) reduces to a twofold one. In this way, we have explicit expressions for the Macdonald polynomials of types $C_n$ and $D_n$ with one column diagrams.

Let $n\in \bbZ_{>0}$ and $x=(x_1,\ldots,x_n)$ be a sequence of variables. Let $P_{(1^r)}(x\,|\,a,b,c,d\,|\,q,t)$ be the Koornwinder polynomial with one column diagram $(1^r)$, $r\in \bbZ_{\geq 0}$. (See Section~\ref{Koornwinder}, as to our notation.)

\begin{dfn}\label{Def-E} Define the symmetric Laurent polynomial $E_r(x)$'s by expanding the generating function $E(x\,|\,y)$ as
\begin{gather*}
E(x\,|\,y)=\prod_{i=1}^{n} (1-y x_i)(1-y/x_i) = \sum_{r \geq 0} (-1)^r E_r(x) y^r.
\end{gather*}
Note that we have $E_{2n-r}(x)=E_{r}(x)$ for $0\leq r\leq n$ and $E_r(x)=0$ for $r>2n$.
\end{dfn}

\begin{thm}\label{KoornwinderPoly_1^r} We have the following fourfold summation formula for the $BC_n$ Koornwinder polynomial $P_{(1^r)}(x\,|\,a,b,c,d\,|\,q,t)$ with one column diagram
\begin{gather*}
P_{(1^r)}(x\,|\,a,b,c,d\,|\,q,t) = \sum_{k,l,i,j\geq 0} (-1)^{i+j} E_{r-2k-2l-i-j} (x)
\widehat{c}\,'_e\big(k,l;t^{n-r+1+i+j}\big) \widehat{c}_o\big(i,j; t^{n-r+1}\big),
\end{gather*}
where
%\begin{subequations}
\begin{gather*}
\widehat{c}\,'_e(k,l;s) = { \big(tc^2/a^2 ; t^2\big)_k \big(sc^2t ; t^2\big)_k \big(s^2c^4/t^2 ; t^2\big)_k
\over \big(t^2 ; t^2\big)_k \big(sc^2/t ; t^2\big)_k \big(s^2a^2c^2/t ; t^2\big)_k }
{ \big(1/c^2 ; t\big)_l (s/t ; t)_{2k+l} \over (t ; t)_l \big(sc^2 ; t\big)_{2k+l} } { 1-st^{2k+2l-1}
\over 1-st^{-1} } a^{2k}c^{2l}, \\
\widehat{c}_o(i,j; s) = { (-a/b ; t)_i (scd/t ; t)_i \over (t ; t)_i (-sac/t ; t)_i } { (s ; t)_{i+j} (-sac/t ; t)_{i+j} \big(s^2a^2c^2/t^3 ; t\big)_{i+j}
\over (s^2abcd/t^2 ; t)_{i+j} (sac/t^{3/2} ; t)_{i+j} \big({-}sac/t^{3/2} ; t \big)_{i+j}} \\
\hphantom{\widehat{c}_o(i,j; s) =}{} \times { (-c/d ; t)_j (sab/t ; t)_j \over (t ; t)_j (-sac/t ; t)_j } b^id^j.\nonumber
\end{gather*}
%\end{subequations}
\end{thm}

\begin{cor}\label{MacdonaldPoly_1^r} Degenerating Koornwinder's parameters as $(a,b,c,d)\rightarrow (a,-a,c,-c)$
we have
\begin{gather}
 P_{(1^r)}(x\,|\,a,-a,c,-c\,|\,q,t) =
 \sum_{k, l \geq 0 \atop
2k+2l \leq r}E_{r-2k-2l}(x) { \big(1/c^2 ; t\big)_l (s/t ; t)_{2k+l} \over (t ; t)_l \big(sc^2 ; t\big)_{2k+l} }
{ 1-st^{2k+2l-1} \over 1-st^{-1} } c^{2l} \nonumber\\
\hphantom{P_{(1^r)}(x\,|\,a,-a,c,-c\,|\,q,t) = }{} \times
 { \big(tc^2/a^2 ; t^2\big)_k \big(sc^2t ; t^2\big)_k \big(s^2c^4/t^2 ; t^2\big)_k
\over \big(t^2 ; t^2\big)_k \big(sc^2/t ; t^2\big)_k \big(s^2a^2c^2/t ; t^2\big)_k }
a^{2k},\label{PtoG_0}
\end{gather}
where $s=t^{n-r+1}$.
\end{cor}

The following formula (\ref{G-to-P}) can be derived from (\ref{PtoG_0}) by applying the Bressoud matrix inversion technique~\cite{B, L}. See Section~\ref{MatrixInverse}, for details.
\begin{thm}\label{THME-P}We have
\begin{gather}
E_{r}(x) = \sum_{k, l \geq 0 \atop 2k+2l \leq r} P_{(1^{r-2l-2k})}(x\,|\,a,-a,c,-c\,|\,q,t)
{ \big(c^2 ; t\big)_l \over (t ; t)_l }{ \big(s t^l ; t\big)_{l+2k} \over \big(st^{l-1}c^2 ; t\big)_{l+2k} } \nonumber\\
\hphantom{E_{r}(x) =}{} \times
{ \big(a^2/tc^2 ; t^2\big)_k \over \big(t^2 ; t^2\big)_k } { \big(s^2 t^{4l-2} c^4 ; t^2\big)_k \over \big(s^2 t^{4l} c^4 ; t^2\big)_k }
{ \big(s^2 t^{4l+2k-2} c^4 ; t^2\big)_k \over \big(s^2 t^{4l+2k-3} a^2c^2 ; t^2\big)_k }\big(tc^2\big)^k,\label{G-to-P}
\end{gather}
where $s=t^{n-r+1}$.
\end{thm}

This is proved in (\ref{e-p}) of Theorem \ref{thm:g_r}.

\subsection[The coefficients $B(s,j)$ and $\widetilde{B}(s,j)$]{The coefficients $\boldsymbol{B(s,j)}$ and $\boldsymbol{\widetilde{B}(s,j)}$}
By using the $q$-analogue of Bailey's transformation \cite[p.~99, equation~(3.10.14)]{GR}, one can rewrite the twofold summations in~(\ref{PtoG_0}) and~(\ref{G-to-P}) as sums having certain~${}_4\phi_3$ series as their coefficients.

\begin{dfn}\label{B-Bt-def} Let $B(s,j)$ and $\widetilde{B}(s,j)$ be the rational functions in $s$ defined by
%\begin{subequations}
\begin{gather*}
B(s,j)=(-1)^j s^{-j} {\big(s^2/t^2; t^2\big)_j \over \big(t^2; t^2\big)_j} {1-s^2 t^{4j-2} \over 1-s^2 t^{-2}}
{}_{4}\phi_3 \left[ { -sa^2, -sc^2, s^2 t^{2j-2}, t^{-2j}
\atop -s, -st, s^2a^2 c^2/t } ; t^2, t^2\right],\\
 \widetilde{B}(s,j)=
\big(st^{j-1}\big)^{-j}
{\big(t^{2j}s^2;t^2\big)_j\over \big(t^2;t^2\big)_j}
 {}_{4}\phi_3 \left[ {-t^{-2j+2}/sa^2, -t^{-2j+2}/sc^2, t^{-2j+2}/s^2 , t^{-2j}
\atop
 -t^{-2j+1}/s, -t^{-2j+2}/s,t^{-4j+5}/s^2a^2c^2} ; t^2, t^2\right].
\end{gather*}
%\end{subequations}
\end{dfn}

\begin{thm}\label{Thm2.6} The formulas \eqref{PtoG_0} and \eqref{G-to-P} can be recast as $($see Theorem~{\rm \ref{BIBASIC})}
\begin{subequations}
\begin{gather}
P_{(1^r)}(x\,|\,a,-a,c,-c\,|\,q,t) = \sum_{j=0}^{\lfloor {r \over 2} \rfloor} B\big(t^{n-r+1},j\big) E_{r-2j}(x),\label{PtoE}\\
E_{r}(x) = \sum_{j=0}^{\lfloor {r \over 2} \rfloor} \widetilde{B}\big(t^{n-r+1},j\big) P_{(1^{r-2j})}(x\,|\,a,-a,c,-c\,|\,q,t).
\end{gather}
\end{subequations}
\end{thm}

\begin{dfn} Let $f(s)$ be the function defined in (\ref{f[s]}). For ease of notation, define
\begin{gather*}
F(s,l)=f\big(s/t^l\big)=\dfrac{\big(1-t^l/s \big)\big(1-t^{l+2}/s a^2c^2 \big)\big(1+t^{l+1}/s a^2\big)\big(1+t^{l+1}/sc^2\big)}
{\big(1-t^{2l+1}/s^{2}a^2 c^2\big)\big(1-t^{2l+3}/s^{2}a^2 c^2\big)}.\label{F[s,j]}
\end{gather*}
\end{dfn}
We summarize the basic properties for the functions $B(s,i)$'s and $\widetilde{B}(s,i)$'s. See Theorem \ref{four-term} and Proposition~\ref{ANOTHER}.
\begin{thm}We have the four term relations
%\begin{subequations}
\begin{gather*}
B(s,i)+F(s,-1)B\big(st^{2},i-1\big) =B(st,i) + B(st,i-1),\\
\widetilde{B}(s,i) +F(s,2-2i) \widetilde{B}(s,i-1)=\widetilde{B}\big(st^{-1},i\big)+ \widetilde{B}(st,i-1),
\end{gather*}
%\end{subequations}
and the inversion relations
%\begin{subequations}
\begin{gather*}
\sum_{k=0}^i B(s,k)\widetilde{B}\big(s t^{2k},i-k\big)=\delta_{i,0},\qquad
\sum_{k=0}^i \widetilde{B}(s,k)B\big(s t^{2k},i-k\big)=\delta_{i,0}.
\end{gather*}
%\end{subequations}
\end{thm}

\subsection[Transition matrix from $E_{r}(x)$ to the $BC_n$ interpolation polynomial $E_r(x;a|t)$]{Transition matrix from $\boldsymbol{E_{r}(x)}$\\ to the $\boldsymbol{BC_n}$ interpolation polynomial $\boldsymbol{E_r(x;a|t)}$}\label{interpolate}

Let $\langle z; w \rangle = z + z^{-1} - w - w^{-1}$.

\begin{dfn}[{\cite[equation~(5.1)]{KNS}}]\label{KNS} Set
\begin{gather*}
E_r(x;a\,|\,t)=\sum_{1\leq i_1<\cdots< i_r\leq n}
\big\langle x_{i_1};t^{i_1-1}a\big\rangle \big\langle x_{i_2};t^{i_2-2}a\big\rangle \cdots \big\langle x_{i_r};t^{i_r-r}a\big\rangle .
\end{gather*}
\end{dfn}

These Laurent polynomials $E_r(x;a\,|\,t)$, $r = 0, 1, \dots, n$, are essentially the $BC_n$ interpolation polynomials of Okounkov~\cite{O} attached to single columns~$(1^r)$.

\begin{thm}[{\cite[Theorem~5.1]{KNS}}]
Let $s=t^{n-r+1}$. We have
\begin{gather*}
P_{(1^r)}(x\,|\,a,b,c,d\,|\,q,t) =\sum_{l=0}^r {(s,s a b/t,s a c/t,s a d/t;t)_l \over t^{l(l-1)/2} (as/t)^l \big(t,s^2 a b c d/t^2;t\big)_l }E_{r-l}(x;a|t).
\end{gather*}
\end{thm}

\begin{thm}\label{inverse-EtoP}We have
\begin{gather*}
E_{r}(x;a|t) =\sum_{l=0}^r (-1)^l {(s,s a b/t,s a c/t,s a d/t;t)_l \over (as/t)^l \big(t,t^{l-3}s^2 a b c d ;t\big)_l }P_{(1^{r-l})}(x|a,b,c,d|q,t).
\end{gather*}
\end{thm}

A proof of this is given by using Krattenthaler's matrix inversion as follows. An infinite-dimensional matrix $(f_{ij})_{i,j \in \mathbb{Z}_{\geq 0}}$ is said to be lower-triangular if $f_{ij}=0$ unless $i \geq j$. Two infinite-dimensional lower-triangular matrices $(f_{ij})_{i,j \in \mathbb{Z}_{\geq 0}}$ and $(g_{ij})_{i,j \in \mathbb{Z}_{\geq 0}}$ are said to be mutually inverse if $\sum\limits_{i \geq j \geq k} f_{ij} g_{jk} = \delta_{i,k}$.

\begin{thm}[\cite{Kr}] Let $\mathcal{N}(x,y;q)$ be the infinite lower-triangle matrix with entries
\begin{gather*}
\mathcal{N}_{i,j}(x,y;q)=y^{i-j} { (x/y;q)_{i-j} \over (q;q)_{i-j} } {1 \over \big(xq^{i+j};q\big)_{i-j} \big(yq^{2j+1};q\big)_{i-j} },
\end{gather*}
Then $\mathcal{N}(x,y;q)$ and $\mathcal{N}(y,x;q)$ are mutually inverse.
\end{thm}

If two infinite matrices $(f_{ij})$ and $(g_{ij})$ are mutually inverse, then the conjugated ones $(f_{ij} d_i/d_j)$ and $(g_{ij} d_i/d_j)$ are also mutually inverse for any sequence $(d_{r})$ with nonzero entries.

\begin{dfn} Let $\vector{u} = (u_1, u_2, u_3, u_4)$. Set
\begin{gather*}
g(\vector{u},v)_r= v^{-r} \big(u_1 t^{-r};t\big)_r (u_2;t)_r (u_3;t)_r (u_4;t)_r.
\end{gather*}
Note that we have
\begin{gather*}
g(\vector{u}, v)_r / g(\vector{u}, v)_{r-i}= v^{-i} \big(u_1 t^{-r};t\big)_i \big(u_2 t^{r-i};t\big)_i \big(u_3 t^{r-i};t\big)_i \big(u_4 t^{r-i};t\big)_i.
\end{gather*}
Let $\widetilde{\mathcal{N}}{( \vector{u}, v, x, y ;t)}$ be the conjugation of the matrix $\mathcal{N}(x,yv;t)$ by the sequence $(g(\vector{u}, v)_r)$ with entries given by
\begin{align*}
\widetilde{\mathcal{N}}_{r,r-i}( \vector{u}, v, x, y ;t)&=\mathcal{N}_{r,r-i}(x,yv;t) \times g(\vector{u}, v)_r / g(\vector{u}, v)_{r-i} \\
&= y^i { (x/yv;t)_i \over (t;t)_i }
{ \big(u_1 t^{-r};t\big)_i \big(u_2 t^{r-i};t\big)_i \big(u_3 t^{r-i};t\big)_i \big(u_4 t^{r-i};t\big)_i
\over \big(x t^{2r-i};t\big)_i \big(yv t^{2r-2i+1};t\big)_i }.
\end{align*}
\end{dfn}

\begin{proof}[Proof of Theorem \ref{inverse-EtoP}]
It follows Krattenthaler's matrix inversion and
\begin{gather*}
\widetilde{\mathcal{N}}_{r,r-l}\big(t^{n-1}, t^{-n+1}/ab, t^{-n+1}/ac, t^{-n+1}/ad, t^{-2n+2}/a^2bcd, a/t, 0 ; t\big)\\
\qquad{} ={(s,s a b/t,s a c/t,s a d/t;t)_l \over t^{l(l-1)/2} (as/t)^l \big(t,s^2 a b c d/t^2;t\big)_l },\\
 \widetilde{\mathcal{N}}_{r,r-l}\big(t^{n-1}, t^{-n+1}/ab, t^{-n+1}/ac, t^{-n+1}/ad, t^{-2n+2}/a^2bcd, 0, a/t ; t\big) \\
\qquad{} =(-1)^l {(s,s a b/t,s a c/t,s a d/t;t)_l \over (as/t)^l \big(t,t^{l-3}s^2 a b c d ;t\big)_l}.\tag*{\qed}
\end{gather*}\renewcommand{\qed}{}
\end{proof}

\begin{thm}
We have
\begin{gather*}
 E_{r}(x;a|t)=\sum_{m=0}^r (-1)^m e\big(t^{n-r+1},m\big)E_{r-m}(x),
\end{gather*}
where
\begin{gather}
e(s,m)=(t/as)^m {(s;t)_{m} \big({-}s a^2 t^{-m};t^2\big)_{m}\over (t;t)_{m}}\nonumber\\
\hphantom{e(s,m)=}{} \times
{}_4\phi_3\left[{t^{-m},t^{-m+1}, -t^{-m+1}/s,-t^{-m+2}/s \atop
-t^{-m+2}/a^2s, -t^{-m}a^2s,t^{-2m+4}/s^2};t^2,t^2 \right]. \label{e(s,m)}
\end{gather}
\end{thm}

\begin{proof}For simplicity of display, we write
\begin{gather}
b(s,i):= (-1)^i s^{-i} {\big(s^2; t^2\big)_i \over (t^2; t^2)_i}{1-s^2 t^{4i-2} \over 1-s^2 t^{2i-2}},\label{b(s,i)}\\
{}_{4}\phi_3(s,i,k) := { \big( {-}sa^2, -sc^2, s^2 t^{2i-2}, t^{-2i} ; t^2\big)_k
\over \big(t^2, -s, -st, s^2a^2 c^2t^{-1} ; t^2\big)_k }.\label{4phi3(s,i,k)}
\end{gather}
Then we have
$B(s,i) = b(s,i) \sum\limits_{k=0}^i {}_{4}\phi_3(s,i,k)$. From Theorems~\ref{Thm2.6} and~\ref{KNS} written for the case $b=-a$ and $d=-c$, we have
\begin{gather*}
E_{r}(x;a\,|\,t) = \sum_{l=0}^r (-1)^l
{\big(s,-s a^2/t,s a c/t,-s ac/t;t\big)_l \over (as/t)^l \big(t,t^{l-3}s^2 a^2c^2 ;t\big)_l }
\sum_{j=0}^{\lfloor {r-l\over 2}\rfloor} B\big(st^{l},j\big)E_{r-l-2j}(x)\\
\hphantom{E_{r}(x;a\,|\,t)}{} =\sum_{m=0}^r (-1)^m e(s,m) E_{r-m}(x),
\end{gather*}
where
\begin{gather*}
e(s,m)=\sum_{j=0}^{\lfloor {m\over 2}\rfloor}
{\big(s,- s a^2/t,s a c/t,-s a c/t;t\big)_{m-2j} \over (as/t)^{m-2j} \big(t,t^{m-2j-3}s^2 a^2c^2 ;t\big)_{m-2j} }B\big(st^{m-2j},j\big).
\end{gather*}
Changing the order of the summation, we have
\begin{gather}
e(s,m)= \sum_{j=0}^{\lfloor {m\over 2}\rfloor}
{\big(s, -s a^2/t,s a c/t,-s ac/t;t\big)_{m-2j} \over (as/t)^{m-2j} \big(t,t^{m-2j-3}s^2 a^2c^2 ;t\big)_{m-2j} }
b\big(st^{m-2j},j\big) \sum_{k=0}^j {}_{4}\phi_3\big(s t^{m-2j},j ,k\big)\nonumber\\
\hphantom{e(s,m)}{} = \sum_{i=0}^{\lfloor {m\over 2}\rfloor}\sum_{k=0}^{\lfloor {m-2i\over 2}\rfloor}
{\big(s,-s a^2/t,s a c/t,-s a c/t;t\big)_{m-2i-2k} \over (as/t)^{m-2i-2k} \big(t,t^{m-2i-2k-3}s^2 a^2c^2 ;t\big)_{m-2i-2k} }\nonumber\\
\hphantom{e(s,m)=}{} \times b\big(st^{m-2i-2k},i+k\big){}_{4}\phi_3\big(s t^{m-2i-2k},i+k ,k\big). \label{e-temp}
\end{gather}
Simplifying the expression, we have
\begin{gather}
 \sum_{k=0}^{\lfloor {m-2i\over 2}\rfloor}
{\big(s,-s a^2/t,s a c/t,-s a c/t;t\big)_{m-2i-2k} \over (as/t)^{m-2i-2k} \big(t,t^{m-2i-2k-3}s^2 a^2c^2 ;t\big)_{m-2i-2k} }\nonumber\\
\qquad\quad{}\times b\big(st^{m-2i-2k},i+k\big){}_{4}\phi_3\big(s t^{m-2i-2k},i+k ,k\big)\nonumber\\
\qquad{} = {\big(s,-s a^2/t,s a c/t,-s a c/t;t\big)_{m-2i} \over (as/t)^{m-2i} \big(t,t^{m-2i-3}s^2 a^2c^2 ;t\big)_{m-2i} } b\big(st^{m-2i},i\big)\label{e-temp2}\\
 \qquad{}\quad\times
{}_6W_5\big(t^{-2m+4i+3}/s^2a^2c^2;-t^{-m+2i+2}/sc^2, t^{-m+2i},t^{-m+2i+1};t^2,- t^{m-2i+2}/a^2s\big).\nonumber
\end{gather}
By using \cite[p.~42, equation~(2.4.2)]{GR}, we have the factorized expression for this ${}_6W_5$-series as
\begin{gather}
{\big(s^2 a^2 c^2 t^{m-2i-3};t\big)_{m-2i}\over \big(s^2 a^2 c^2/t^2;t^2\big)_{m-2i}}
{\big({-}s a^2 t^{-m+2i};t^2\big)_{m-2i}\over \big({-}s a^2/t;t\big)_{m-2i}}.\label{e-temp3}
\end{gather}
Then simplifying the factors again, we have (\ref{e(s,m)}) from (\ref{e-temp}), (\ref{e-temp2}) and (\ref{e-temp3}).
\end{proof}

\subsection[The coefficients $C(s,j)$ and Catalan triangle three term relations]{The coefficients $\boldsymbol{C(s,j)}$ and Catalan triangle three term relations}\label{basic-C}

We have (in Lemma \ref{Lem-Em} below) the expansion of $E_r(x)$ in terms of the monomial symmetric polynomials as
\begin{gather}
E_r(x) = \sum_{j=0}^{\lfloor{r \over 2}\rfloor} \binom{n-r+2j}{j} m_{(1^{r-2j})}(x), \label{E-m}
\end{gather}
where $\binom{m}{j}$ denotes the ordinary binomial coefficient. In view of this, we are naturally led to the following definition.
\begin{dfn}\label{C(s,j)}Let $s=t^{m+1}$ for $m \in \mathbb{C}$ and $C(s,j)$ be the function of $s$ defined by
\begin{gather*}
C(s,j):=\sum_{i=0}^{j} B(s,i) \binom{m+2j}{j-i}.
\end{gather*}
\end{dfn}

\begin{thm}\label{Pvsm} The Koornwinder polynomial $P_{(1^r)}(x\,|\,a,-a,c,-c\,|\,q,t)$ is expanded in terms of the monomial symmetric polynomials as
\begin{gather*}
P_{(1^r)}(x\,|\,a,-a,c,-c\,|\,q,t) = \sum_{j=0}^{\lfloor {r \over 2} \rfloor} C\big(t^{n-r+1},j\big) m_{(1^{r-2j})}(x).
\end{gather*}
\end{thm}

\begin{proof} It follows from (\ref{PtoE}) and (\ref{E-m}).
\end{proof}

\begin{thm} \label{threeterm-C}We have the three term relation $($see Proposition~{\rm \ref{THREE-1})}
\begin{gather}
 C(s,j) + F(s, -1) C\big(st^2,j-1\big)= C(st,j), \label{CCC}
\end{gather}
and the specialization formula for $s=1$, i.e., for $m=-1$ $($see Proposition~{\rm \ref{THREE-2})}
\begin{gather}
C(1,j)=\delta_{j,0}.\label{SPE}
\end{gather}
\end{thm}

\begin{dfn}Define the infinite upper triangular matrix $\mathcal{C}=(\mathcal{C}_{ij})_{i,j\in\mathbb{Z}_{\geq 0}} $ by setting for $r,i\geq 0$
\begin{gather*}
\mathcal{C}_{r,r+2i} =C\big(t^{r+1},i\big), \qquad \mathcal{C}_{r,r+2i+1}=0.
\end{gather*}
\end{dfn}

\begin{thm}The $\mathcal{C}_{ij}$'s satisfy the recursion relations in Theorem~{\rm \ref{MAIN}}
\begin{subequations}
\begin{gather}
\mathcal{C}_{0,0}=1,\qquad \mathcal{C}_{i-1,i-1}=\mathcal{C}_{i,i}, \qquad i=1,2,3,\ldots,\label{CAT2}\\
F(1,-1)\mathcal{C}_{1,j-1}=\mathcal{C}_{0,j}, \qquad j=2,4,6,\ldots,\label{CAT3}\\
\mathcal{C}_{i-1,j-1}+F\big( t^i,-1\big)\mathcal{C}_{i+1,j-1}=\mathcal{C}_{i,j}, \qquad i+j \ {\rm even},\quad 0<i< j.\label{CAT4}
\end{gather}
\end{subequations}
\end{thm}
\begin{proof} We have $\mathcal{C}_{0,0}=1$. When $i+j$ is even and $0\leq i\leq j$, we have $\mathcal{C}_{ij}=C\big(t^{i+1},(j-i)/2\big)$. Therefore from (\ref{CCC}) we have
\begin{gather}
\mathcal{C}_{i-1,j-1}+F\big( t^i,-1\big)\mathcal{C}_{i+1,j-1} = C\big(t^{i},(j-i)/2\big)+F\big( t^i,-1\big)C\big(t^{i+2},(j-i)/2-1\big) \nonumber\\
\hphantom{\mathcal{C}_{i-1,j-1}+F\big( t^i,-1\big)\mathcal{C}_{i+1,j-1}}{} =C\big(t^{i+1},(j-i)/2\big)=\mathcal{C}_{ij}, \label{CCC-check}
\end{gather}
giving the three term recursion relation (\ref{CAT4}) for $0<i< j$ in (\ref{CCC-check}). When $0<i=j$, noting that $\mathcal{C}_{i+1,i-1}=0$ from the upper triangularity, we have~(\ref{CAT2}). When $i=0$ and $j\in2\mathbb{Z}_{>0}$, we have from~(\ref{SPE}) that $\mathcal{C}_{-1,j-1}=C(1,(j-2)/2)=0$, hence (\ref{CAT3}) holds.
\end{proof}

\subsection{Proof of the main theorem}\label{ProofMAIN}
Now we are ready to present a proof of our main theorem.

\begin{proof}[Proof of Theorem \ref{MAIN}] The transition matrix $\mathcal{C}$ is even and upper triangular. In view of Theorem \ref{Pvsm} and $C\big(t^{n-r+1},j\big)={\mathcal C}_{n-r,n-r+2j}$, we have for any $n>0$ and $0\leq r\leq n$
\begin{gather*}
P_{(1^r)}(x\,|\,a,-a,c,-c\,|\,q,t) = \sum_{j=0}^{\lfloor {r \over 2} \rfloor}{\mathcal C}_{n-r,n-r+2j}
m_{(1^{r-2j})}(x),
\end{gather*}
indicating the stabilized transition formula (\ref{STAB}). The three term recursion relation (\ref{Catalan-1}), (\ref{Catalan-2}) and
(\ref{Catalan-3}) are shown in Theorem~\ref{threeterm-C}.
\end{proof}

\section[Koornwinder's $q$-difference operator, Koornwinder polynomials and Mimachi's kernel function]{Koornwinder's $\boldsymbol{q}$-difference operator,\\ Koornwinder polynomials and Mimachi's kernel function}\label{Koornwinder}

We briefly recall some basic properties concerning the Koornwinder polynomials \cite{K} and the Mimachi's kernel function identity \cite{Mi}.

\subsection{Koornwinder's operator and Mimachi's kernel function}\label{Kernel}

Let $a$, $b$, $c$, $d$, $q$, $t$ be complex parameters. We assume that $|q|<1$. Set $\alpha=(abcd/q)^{1/2}$ for simplicity. Let $x=(x_1,\ldots,x_n)$ be a sequence of independent indeterminates. The Weyl group of type $BC_n$ is denoted by $W_n (\simeq \mathbb{Z}_2^n \rtimes \mathfrak{S}_n)$. Let $\mathbb{C}\big[x_1^{\pm}, x_2^{\pm}, \ldots, x_n^{\pm}\big]^{W_n}$ be the ring of $W_n$-invariant Laurent polynomials in~$x$. For a partition $\lambda = (\lambda_1, \lambda_2, \ldots, \lambda_n)$ of length $n$, i.e., $\lambda_i\in \bbZ_{\geq 0}$ and $\lambda_1\geq \cdots\geq \lambda_n$, we denote by $m_{\lambda}=m_{\lambda}(x)$ the monomial symmetric polynomial being defined as the orbit sums of monomials
\begin{gather*}
m_{\lambda} = {1 \over |\operatorname{Stab}(\lambda)|} \sum_{\mu \in W_n \cdot \lambda} \prod_{i} x_i^{\mu_i},
\end{gather*}
where $\operatorname{Stab}(\lambda)=\{ s \in W_n \,|\, s \lambda = \lambda \}$.

Koornwinder's $q$-difference operator ${\mathcal D}_x={\mathcal D}_x(a,b,c,d\,|\,q,t)$ \cite{K} reads
\begin{gather*}
{\mathcal D}_x= \sum_{i=1}^n {(1-ax_i)(1-bx_i)(1-cx_i)(1-dx_i)\over
\alpha t^{n-1}\big(1-x_i^2\big)\big(1-qx_i^2\big)} \prod_{j\neq i} {(1-t x_ix_j)(1-t x_i/x_j)\over (1-x_ix_j)(1-x_i/x_j)} \big(T_{q,x_i}^{+1}-1\big) \\
{}+ \sum_{i=1}^n {(1-a/x_i)(1-b/x_i)(1-c/x_i)(1-d/x_i)\over \alpha t^{n-1}\big(1-1/x_i^2\big)\big(1-q/x_i^2\big)} \prod_{j\neq i} {(1-t x_j/x_i)(1-t /x_ix_j)\over (1-x_j/x_i)(1-1/x_ix_j)} \big(T_{q,x_i}^{-1}-1\big),
\end{gather*}
where we have used the notation $T_{q,x}^{\pm1}f(x_1,\ldots,x_i,\ldots ,x_n)=f\big(x_1,\ldots,q^{\pm 1}x_i,\ldots ,x_n\big)$.

The Koornwinder polynomial $P_\lambda(x)=P_\lambda(x\,|\,a,b,c,d\,|\,q,t)\in \mathbb{C}\big[x_1^{\pm 1},\ldots,x_n^{\pm 1}\big]^{W_n}$
is uniquely characterized by the conditions
\begin{gather*}
(a) \quad \mbox{$P_\lambda(x)=m_\lambda(x)+\mbox{lower order terms}$ w.r.t.\ the dominance ordering}, \\
(b) \quad \mbox{${\mathcal D}_x P_\lambda=d_\lambda P_\lambda$.}
\end{gather*}
The eigenvalue $d_\lambda$ is explicitly written as
\begin{gather*}
 d_\lambda=\sum_{j=1}^n\big\langle abcdq^{-1}t^{2n-2j}q^{\lambda_j}\big\rangle
\big\langle q^{\lambda_j}\big\rangle =
\sum_{j=1}^n \big\langle \alpha t^{n-j}q^{\lambda_j}; \alpha t^{n-j} \big\rangle,
\end{gather*}
where we used the notations $\langle x\rangle=x^{1/2}-x^{-1/2}$ and $\langle x;y\rangle=\langle xy\rangle\langle x/y\rangle=x+x^{-1}-y-y^{-1}$
for simplicity of display.

\begin{dfn} Define the involution $\widetilde{*}$ of the parameters by
\begin{gather*}
\widetilde{a}=a, \qquad
\widetilde{b}=b, \qquad
\widetilde{c}=c, \qquad
\widetilde{d}=d, \qquad
\widetilde{q}=t,\qquad
\widetilde{t}=q.
\end{gather*}
We write $\widetilde{\mathcal D}_x={\mathcal D}_x( a, b, c, d \,|\, t, q)$ and $\widetilde{P}_\lambda(x)=P_\lambda(x\,|\, a, b, c, d \,|\, t, q)$ for short.
\end{dfn}

\begin{thm}[{\cite[Lemma 3.2]{Mi}}]\label{MIMACHI} Let $n$ and $m$ be positive integers, and let $x=(x_1,\ldots,x_n)$, $y=(y_1,\ldots,y_m)$ be two sets of independent indeterminates. Mimachi's kernel function
\begin{gather*}
\Psi(x;y)=(y_1 y_2 \cdots y_m)^{-n} \prod_{i=1}^n \prod_{j=1}^m (1-y_j x_i)(1-y_j/x_i),
%\label{Kernel-function}
\end{gather*}
enjoys the kernel function identity
\begin{gather*}
\langle t \rangle {\mathcal D}_x \Psi(x;y)+ \langle q \rangle \widetilde{{\mathcal D}}_y \Psi(x;y)=
\big\langle t^n \big\rangle \big\langle q^m \big\rangle \big\langle abcdt^{n-1}q^{m-1} \big\rangle \Psi(x;y).
\end{gather*}
\end{thm}

When we apply Mimachi's kernel function, the following lemmas will be used. Recall that the generating function $E(x|y)$ is introduced in Definition~\ref{Def-E}
\begin{gather*}
E(x\,|\,y)=\prod_{i=1}^{n} (1-y x_i)(1-y/x_i) = \sum_{r \geq 0} (-1)^r E_r(x) y^r.
\end{gather*}

\begin{lem}\label{Lem-Em} We have
\begin{gather*}
E_r(x) = \sum_{k=0}^{\lfloor{r \over 2}\rfloor} \binom{n-r+2k}{k} m_{(1^{r-2k})}(x),
\end{gather*}
where $\binom{m}{j}$ denotes the ordinary binomial coefficient.
\end{lem}
\begin{proof}
For an integer $s$ satisfying $0\leq s\leq n$, we can find that the coefficient of the monomial $x_1x_2\cdots x_s$ in $E(x\,|\,y)=\prod\limits_{i=1}^{n} \big(1-(x_i+1/x_i)y +y^2\big)$ is $(-1)^s y^s \big(1+y^2\big)^{n-s}$. Hence we have
\begin{gather*}
E(x\,|\,y) =\sum_{s=0}^n \sum_{k=0}^{n-s} m_{(1^s)}(x) (-1)^s y^{s+2k} \binom{ n-s }{k} \\
\hphantom{E(x\,|\,y)}{} =\sum_{r=0}^n(-1)^r y^r\sum_{k=0}^{\lfloor{ r\over 2} \rfloor}
 \binom{ n-r+2k }{ k }m_{(1^{r-2k})}(x). \tag*{\qed}
\end{gather*}\renewcommand{\qed}{}
\end{proof}

\begin{lem}\label{Lem-EEE} Let $\lambda=(\lambda_1,\ldots,\lambda_m)$ be a partition satisfying the condition $\lambda_1\leq n$. We have
\begin{gather}
\prod_{i=1}^m E_{\lambda_i}(x)=m_{\lambda'}(x)+\mbox{ lower order terms}. \label{EEEE}
\end{gather}
\end{lem}
\begin{proof} Note that for any partitions $\lambda$ and $\mu$, we have $m_\lambda m_\mu=m_{\lambda+\mu} +\mbox{lower order terms}$.
By using Lemma \ref{Lem-Em}, we have $E_r(x)=m_{(1^r)}+\mbox{lower order terms}$.
Hence we have~(\ref{EEEE}).
\end{proof}

\subsection[Asymptotically free eigenfunction $f(x;s)$ for ${\mathcal D}_x$ and reproduction formula]{Asymptotically free eigenfunction $\boldsymbol{f(x;s)}$ for $\boldsymbol{{\mathcal D}_x}$\\ and reproduction formula}\label{reproduction}

Let $\lambda=(\lambda_1,\ldots,\lambda_n)\in \mathbb{C}^n$ and $s=(s_1,\ldots,s_n)$ be a sequence of complex parameters as $s_i=t^{-n+i}q^{-\lambda_i}$, $i=1,\ldots,n$. We use the shorthand notations $x^{\lambda}=\prod_ix_i^{\lambda_i}$ and $x^{-\lambda}=\prod_ix_i^{-\lambda_i}$. Let $f(x;s)\in x^{-\lambda} \bbC[[x_1/x_2,\ldots,x_{n-1}/x_n,x_n]]$ be the infinite series satisfying the conditions
\begin{gather*}
f(x;s)=x^{-\lambda}\sum_{\beta \in Q^+} c_\beta(s)x^\beta,\qquad c_0(s)=1, \\
{\mathcal D}_x f(x;s)= \sum_{i=1}^n \big\langle \alpha s_i^{-1} ;\alpha t^{n-i}\big\rangle f(x;s),
\end{gather*}
where $Q^+$ denotes the positive cone of the root lattice of type $BC_n$. To be more explicit, corresponding to the simple roots $\alpha_1=\varepsilon_1-\varepsilon_2,\alpha_2=\varepsilon_2-\varepsilon_2,\ldots, \alpha_{n-1}=\varepsilon_{n-1}-\varepsilon_n,\alpha_{n}=\varepsilon_{n}$, we have $x^{\alpha_1}=x_1/x_2,\ldots,x^{\alpha_{n-1}}=x_{n-1}/x_n, x^{\alpha_n}=x_n$. Assuming the genericity of the eigenvalue, one can show that $f(x;s)$ is determined uniquely.

\begin{dfn}The adjoint ${\mathcal D}_x^*$ of ${\mathcal D}_x$ is defined to be
\begin{gather*}
{\mathcal D}_x^*= \sum_{i=1}^n \big(T_{q,x_i}^{-1}-1\big) {(1-ax_i)(1-bx_i)(1-cx_i)(1-dx_i)\over
\alpha t^{n-1}\big(1-x_i^2\big)\big(1-qx_i^2\big)} \prod_{j\neq i} {(1-t x_ix_j)(1-t x_i/x_j)\over (1-x_ix_j)(1-x_i/x_j)} \\
{}+ \sum_{i=1}^n \big(T_{q,x_i}^{+1}-1\big){(1-a/x_i)(1-b/x_i)(1-c/x_i)(1-d/x_i)\over
\alpha t^{n-1}\big(1-1/x_i^2\big)\big(1-q/x_i^2\big)} \prod_{j\neq i} {(1-t x_j/x_i)(1-t /x_ix_j)\over (1-x_j/x_i)(1-1/x_ix_j)}.
\end{gather*}
\end{dfn}
\begin{dfn} Denote by $V(x)$ the Weyl denominator of type~$C_n$
\begin{gather*}
V(x)= \prod_{k=1}^n x_k^{-n+k-1} \prod_{i=1}^n \big(1-x_i^{2}\big) \prod_{1\leq i<j\leq n}(1-x_ix_j)(1-x_i/x_j).
\end{gather*}
\end{dfn}

\begin{dfn}Define the involution $\overline{*}$ of the parameters by
\begin{gather*}
\overline{a}=q/a,\qquad
\overline{b}=q/b,\qquad
\overline{c}=q/c,\qquad
\overline{d}=q/d,\qquad
\overline{q}=q,\qquad
\overline{t}=q/t. \nonumber
\end{gather*}
Write for simplicity the composition of the two involutions $\widetilde{*}$ and $\overline{*}$ as $\widehat{*}=\widetilde{\overline{*}}$. Note that $\widehat{*}$ is not an involution but has order~$6$
\begin{gather*}
\widehat{a}=t/a,\qquad
\widehat{b}=t/b,\qquad
\widehat{c}=t/c,\qquad
\widehat{d}=t/d,\qquad
\widehat{q}=t,\qquad
\widehat{t}=t/q.
\end{gather*}
\end{dfn}

\begin{prp}[{\cite[Proposition 6.2]{HNS}}]\label{ADJOINT} We have
\begin{gather*}
V(x)^{-1}{\mathcal D}_x^*V(x)- \overline{\mathcal D}_x =\sum_{j=1}^n \big\langle\overline{\alpha}\bar{t}^{n-j};\alpha t^{n-j}\big\rangle.
 \end{gather*}
\end{prp}

\begin{thm}\label{thm:reproduction}Let $n\geq m$ be positive integers, and $x=(x_1,\ldots,x_n), y=(y_1,\ldots,y_m)$ be sequences of independent indeterminates. Let $\lambda=(\lambda_1,\lambda_2,\ldots)$ be a partition
satisfying $\ell(\lambda)\leq m$ and $\lambda_1\leq n$. Set
\begin{gather}
s_i=\widehat{t}^{\,\,-m+i}\,\widehat{q}^{\,\,-\lambda_{m+1-i} +m+1-i+n },\qquad 1\leq i\leq m .\label{choice-s}
\end{gather}
Let $\widehat{f}(y;s)$ be the formal series in $y$ uniquely characterized by $\widehat{c}_0(s)=1$ and
\begin{gather*}
 \widehat{f}(y;s)=\prod_{i=1}^m y_i^{-\lambda_{m+1-i} +m+1-i+n }\sum_{\beta\in Q^+} \widehat{c}_\beta(s) y^\beta, \\
\widehat{\mathcal D}_y \widehat{f}(y;s)=\sum_{i=1}^m \big\langle \widehat{\alpha}s_i^{-1};\widehat{\alpha} \widehat{t}^{m-i}\big\rangle \widehat{f}(y;s).
\end{gather*}
Then we have
\begin{gather}
P_{\lambda'}(x\,|\,a,b,c,d\,|\,q,t)=(-1)^{|\lambda|}\big[\Psi(x;y)V(y) \widehat{f}(y;s)\big]_{1,y}, \label{PUT}
\end{gather}
where the notation $[\cdots]_{1,y}$ denotes the constant term in $y$, and $\lambda'$ is the conjugate diagram of $\lambda$.
\end{thm}

\begin{proof} Firstly, we show that the product $\Psi(x;y)V(y) \widehat{f}(y;s)$ has a non-vanishing constant term in~$y$. Write
\begin{gather*}
 \prod_{i=1}^m \big(1-y_i^{2}\big)\cdot
\prod_{1\leq i<j\leq m}(1-y_iy_j)(1-y_i/y_j)\cdot \sum_{\beta\in Q^+} \widehat{c}_\beta(s) y^\beta=\sum_{\beta\in Q^+}\widehat{c}\,'_\beta(s) y^\beta,
\end{gather*}
for short. Noting that we have $\ell (\lambda')\leq n$ from the assumption $\lambda_1\leq n$, we have
\begin{gather*}
 \big[\Psi(x;y)V(y) \widehat{f}(y;s)\big]_{1,y} =\left[\prod_{i=1}^m y_i^{-\lambda_{m+1-i}}\prod_{i=1}^m E(x\,|\,y_i)\cdot \sum_{\beta\in Q^+}\widehat{c}\,'_\beta(s) y^\beta\right]_{1,y} \\
\qquad{} =(-1)^{|\lambda|}\sum_{\beta=\sum k_i\alpha_i \in Q^+} (-1)^{\sum k_i}\widehat{c}\,'_\beta(s) E_{\lambda_{m}-k_{1}}(x)\prod_{i=2}^{m} E_{\lambda_{m+1-i}+k_{i-1}-k_{i}}(x)\\
\qquad{} =(-1)^{|\lambda|} m_{\lambda'}(x) +\mbox{lower order terms}\neq 0.
\end{gather*}
In the last step, we have used Lemma \ref{Lem-EEE}.

Next, we can show that the constant term satisfies the eigenvalue equation as
\begin{gather*}
 \left( {\mathcal D}_x - {\langle t^n\rangle \langle q^m\rangle \langle abcdt^{n-1}q^{m-1}\rangle \over \langle t\rangle}\right) \big[ \Psi(x;y)V(y) \widehat{f}(y;s)\big]_{1,y}\\
\qquad{} = \left[ \left( -\dfrac{\langle q \rangle}{\langle t \rangle} \widetilde{{\mathcal D}}_y\Psi(x;y)\right)V(y) \widehat{f}(y;s)\right]_{1,y}
=-\dfrac{\langle q \rangle}{\langle t \rangle}\big[ \Psi(x;y) \big(\widetilde{{\mathcal D}}_y ^* V(y) \widehat{f}(y;s) \big)\big]_{1,y}\\
\qquad{} = -\dfrac{\langle q \rangle}{\langle t \rangle}\left[ \Psi(x;y) V(y)\left(
\left(\widehat{{\mathcal D}}_y +\sum_{i=1}^m \langle \widehat{\alpha} \widehat{t}^{m-i};\widetilde{\alpha} \widetilde{t}^{m-i} \rangle \right)\widehat{f}(y;s) \right)\right]_{1,y} \\
\qquad{} = -\dfrac{\langle q \rangle}{\langle t \rangle} \left(\sum_{i=1}^m\langle \widehat{\alpha}s_i^{-1};\widehat{\alpha}\widehat{t}^{m-i}\rangle + \langle \widehat{\alpha} \widehat{t}^{m-i};\widetilde{\alpha} \widetilde{t}^{m-i} \rangle \right)\big[ \Psi(x;y) V(y)\widehat{f}(y;s) \big]_{1,y}\\
\qquad{} =-\dfrac{\langle q \rangle}{\langle t \rangle}\left(\sum_{i=1}^m\langle \widehat{\alpha}s_i^{-1} ;\widetilde{\alpha} \widetilde{t}^{m-i} \rangle\right)\big[ \Psi(x;y) V(y)\widehat{f}(y;s) \big]_{1,y}.
\end{gather*}
Here we have used Theorem \ref{MIMACHI}, Proposition \ref{ADJOINT}, and the property $\langle x;y\rangle+\langle y;z\rangle=\langle x;z\rangle$.

To check that the eigenvalue is the desired one, we prepare some lemmas.
\begin{lem}\label{Le-1}The eigenvalue of the Koornwinder polynomial $P_{\lambda'}(x)$ can be recast as
\begin{gather*}
 \sum_{l=1}^{m} \sum_{i=\lambda_{l+1}+1}^{\lambda_l} \big\langle \alpha t^{n-i}q^l ; \alpha t^{n-i} \big\rangle
= \sum_{l=1}^m \dfrac{1}{\langle t \rangle} \big\langle q^l \big\rangle \big\langle t^{\lambda_{l}-\lambda_{l+1}} \big\rangle \big\langle \alpha^2 q^l t^{2n-1-\lambda_l -\lambda_{l+1}} \big\rangle \\
\qquad{}= \dfrac{\langle q^m \rangle}{\langle t \rangle} \big( \alpha q^{m/2} t^{-1/2+n} + \alpha^{-1} q^{-m/2} t^{1/2-n} \big)\\
\qquad\quad{} -\dfrac{\langle q \rangle}{\langle t \rangle} \sum_{l=1}^m \big( \alpha q^{l-1/2} t^{-1/2+n-\lambda_l} + \alpha^{-1} q^{-l+1/2} t^{1/2-n+\lambda_l} \big).
\end{gather*}
\end{lem}

\begin{lem}\label{Le-2} We have
\begin{gather*}
 {1\over \langle t\rangle} \big\langle t^n\big\rangle\big\langle q^{m}\big\rangle \big\langle abcdt^{n-1}q^{m-1}\big\rangle
+\sum_{l=1}^m \dfrac{\langle q \rangle}{\langle t \rangle} \big( \alpha q^{1/2+m-l} t^{-1/2} + \alpha^{-1} q^{-1/2-m+l} t^{1/2} \big) \\
\qquad{} = \dfrac{\langle q^m \rangle}{\langle t \rangle}\big( \alpha q^{m/2} t^{-1/2+n} + \alpha^{-1} q^{-m/2} t^{1/2-n} \big).
\end{gather*}
\end{lem}

Using Lemmas \ref{Le-1} and \ref{Le-2}, and by noting $\widetilde{\alpha}=\alpha q^{1/2} t^{-1/2}$, $\widehat{\alpha}= \alpha^{-1}q^{-1/2}t^{3/2} $ and~(\ref{choice-s}), we can show that
\begin{gather*}
 {1\over \langle t\rangle} \big\langle t^n\big\rangle\big\langle q^{m}\big\rangle \big\langle abcdt^{n-1}q^{m-1}\big\rangle -\dfrac{\langle q \rangle}{\langle t \rangle}
\sum_{i=1}^m \big\langle \widehat{\alpha}s_i^{-1};\widetilde{\alpha} \widetilde{t}^{m-i} \big\rangle \\
\qquad{} = \sum_{l=1}^{m} \sum_{i=\lambda_{l+1}+1}^{\lambda_l}\big\langle \alpha t^{n-i}q^l ; \alpha t^{n-i} \big\rangle
+\dfrac{\langle q \rangle}{\langle t \rangle} \sum_{l=1}^m \big( \alpha q^{l-1/2} t^{-1/2+n-\lambda_l} + \alpha^{-1} q^{-l+1/2} t^{1/2-n+\lambda_l} \big)\\
\qquad\quad{} - \sum_{l=1}^m \dfrac{\langle q \rangle}{\langle t \rangle}\big( \alpha q^{1/2+m-l} t^{-1/2} + \alpha^{-1} q^{-1/2-m+l} t^{1/2} \big)
-\dfrac{\langle q \rangle}{\langle t \rangle} \sum_{i=1}^m \big\langle \widehat{\alpha}s_i^{-1};\widetilde{\alpha} \widetilde{t}^{m-i} \big\rangle \\
\qquad {} = \sum_{l=1}^{m} \sum_{i=\lambda_{l+1}+1}^{\lambda_l} \big\langle \alpha t^{n-i}q^l ; \alpha t^{n-i} \big\rangle.
\end{gather*}
Therefore we have
\begin{gather*}
 \big[\Psi(x;y)V(y) \widehat{f}(y;s)\big]_{1,y} =(-1)^{|\lambda|} m_{\lambda'}(x) +\mbox{lower order terms} ,\\
{\mathcal D}_x \big[ \Psi(x;y)V(y) \widehat{f}(y;s)\big]_{1,y}=\sum_{i=1}^m \big\langle \alpha t^{n-i}q^{\lambda'_i};\alpha t^{n-i}\big\rangle \big[ \Psi(x;y)V(y) \widehat{f}(y;s)\big]_{1,y},
\end{gather*}
thereby proving $P_{\lambda'}(x\,|\,a,b,c,d\,|\,q,t)=(-1)^{|\lambda|} \big[\Psi(x;y)V(y) \widehat{f}(y;s)\big]_{1,y} $.
\end{proof}

\subsection[Macdonald polynomials of types $(C_n,C_n)$, $(C_n,B_n)$ and $(D_n,D_n)$]{Macdonald polynomials of types $\boldsymbol{(C_n,C_n)}$, $\boldsymbol{(C_n,B_n)}$ and $\boldsymbol{(D_n,D_n)}$}\label{notation-C}
We consider some degenerations of the Koornwinder polynomials to the Macdonald polynomials. As for the details, we refer the readers to \cite{K, M1, St}.

Setting the parameters as $(a,b,c,d;q,t)\rightarrow \big(b^{1/2}, -b^{1/2}, q^{1/2}b^{1/2}, -q^{1/2}b^{1/2};q,t\big)$ in the Koornwinder polynomial $P_\lambda(x)$, we obtain the Macdonald polynomials of type $(C_n,C_n)$
\begin{gather*}
P^{(C_n,C_n)}_\lambda(x\,|\,b;q,t)=P_\lambda\big(x\,|\, b^{1/2},-b^{1/2}, q^{1/2}b^{1/2},-q^{1/2}b^{1/2}\,|\,q,t\big).
\end{gather*}
We obtain the Macdonald polynomials of type $(C_n,B_n)$ as
\begin{gather*}
P^{(C_n,B_n)}_\lambda(x\,|\,b;q,t)=P_\lambda\big(x\,|\, b^{1/2},-b^{1/2}, q^{1/2},-q^{1/2}\,|\,q,t\big),
\end{gather*}
and the Macdonald polynomials of type $(D_n,D_n)$ as
\begin{gather*}
P^{(D_n,D_n)}_\lambda(x\,|\,q,t)=P_\lambda\big(x\,|\,1,-1, q^{1/2},-q^{1/2}\,|\,q,t\big).
\end{gather*}
Note that $P^{(D_n,D_n)}_\lambda(x\,|\,q,t)=P^{(C_n,C_n)}_\lambda(x\,|\,1;q,t)= P^{(C_n,B_n)}_\lambda(x\,|\,1;q,t)$.

Note that setting the parameters as $(a,b,c,d;q,t)\rightarrow (a,-a,c,-c;q,t)$ and the application of the twist $\widehat{*}$ on $(a,-a,c,-c;q,t)$ gives
\begin{gather}
(t/a,-t/a, c/t, -c/t; t, t/q). \label{param-C}
\end{gather}

\section{Koornwinder polynomial with one column diagram}\label{one-column}
When we apply Theorem~\ref{thm:reproduction} to the simplest case $m=1$, namely when we plug the $BC_1$ asymptotically free eigenfunction $\widehat{f}(y;s)$ into the formula~(\ref{PUT}), we have the Koornwinder polyno\-mials~$P_{(1^r)}(x)$ with one column diagrams. Note that in $m=1$ case, $\widehat{f}(y;s)$ does not have the para\-me\-ter~$t$. To execute the explicit calculation based on this, we need to recall the fourfold series expansion of the Askey--Wilson polynomials~\cite{HNS}.

Let $D$ denote the Askey--Wilson $q$-difference operator~\cite{AW}
\begin{gather*}
D ={(1-ax)(1-bx)(1-cx)(1-dx)\over \big(1-x^2\big)\big(1-q x^2\big)} \big(T_{q,x}^{+1}-1\big) \\
\hphantom{D =}{} +{(1-a/x)(1-b/x)(1-c/x)(1-d/x)\over \big(1-1/x^2\big)\big(1-q /x^2\big)} \big(T_{q,x}^{-1}-1\big).
\end{gather*}
Let $s\in \bbC$ be a parameter. Introduce $\lambda$ satisfying $s=q^{-\lambda}$. Then we have $T_{q,x} x^{-\lambda} =s x^{-\lambda}$. Let $f(x;s)=f(x;s\,|\,a,b,c,d\,|\,q)$ be a formal series in~$x$
\begin{gather*}
 f(x;s)=x^{-\lambda} \sum_{n\geq 0} c_n x^{n},\qquad c_0\neq 0,
\end{gather*}
satisfying the $q$-difference equation
\begin{gather}
 D f(x;s)=\left(s+{abcd\over qs}-1-{abcd\over q}\right)f(x;s). \label{Df}
\end{gather}
With the normalization $c_0=1$, (\ref{Df}) determines the coefficients $c_n=c_n(s\,|\,a,b,c,d\,|\,q)$ uniquely as rational functions in $a$, $b$, $c$, $d$, $q$ and $s$. We call $f(x;s)=f(x;s\,|\,a,b,c,d\,|\,q)$ the asymptotically free eigenfunction associated with the Askey--Wilson operator~$D$.
\begin{dfn}[{\cite[Definition 3.1]{HNS}}] Set
\begin{gather*}
\Phi(x;s\,|\,a,b,c,d\,|\,q)= \sum_{k,l,m,n\geq 0} c_e\big(k,l;q^{m+n}s\,|\,a,c\,|\,q\big) c_o(m,n;s\,|\,a,b,c,d\,|\,q)x^{2k+2l+m+n}, \nonumber
\end{gather*}
where
%\begin{subequations}
\begin{gather*}
c_e(k,l;s) ={\big(q a^2/c^2;q^2\big)_k \big(q^3 s/c^2;q^2\big)_k \big(q^2 s^2/c^4;q^2\big)_k \over \big(q^2;q^2)_k(q s/c^2;q^2\big)_k\big(q^3 s^2/a^2 c^2;q^2\big)_k}\big(q^2/a^2\big)^k \nonumber \\
\hphantom{c_e(k,l;s) =}{} \times {\big(c^2/q;q\big)_l (s;q)_{2k+l} \over (q;q)_l \big(q^2 s/c^2;q\big)_{2k+l} }\big(q^2/c^2\big)^l, \\
c_o(m,n;s) ={(-b/a;q)_m(qs/cd;q)_m\over (q;q)_m(-qs/ac;q)_m} \\
\hphantom{c_o(m,n;s) =}{} \times {(s;q)_{m+n}(-qs/ac;q)_{m+n}\big(qs^2/a^2c^2;q\big)_{m+n}\over \big(q^2s^2/abcd;q\big)_{m+n}\big(q^{1/2}s/ac ;q\big)_{m+n}\big({-}q^{1/2}s/ac ;q\big)_{m+n}} (q/b)^m \\
\hphantom{c_o(m,n;s) =}{} \times {(-d/c;q)_n (qs/ab;q)_n\over (q;q)_n (-qs/ac;q)_n}(q/d)^n.
\end{gather*}
%\end{subequations}
\end{dfn}

\begin{thm}[{\cite[Theorem 1.2, Proposition 4.3]{HNS}}] The asymptotically free eigenfunction $f(x; s)$ associated with the Askey--Wilson operator~$D$ is expressed as the following fourfold summation
\begin{gather*}
f(x; s)= x^{-\lambda}\Phi(x;s\,|\,a,b,c,d\,|\,q).
\end{gather*}
\end{thm}

\begin{lem}[{\cite[Lemma 5.1]{HNS}}] \label{1-y^2phi} We have
\begin{gather*}
 \big(1-x^2\big)\sum_{k,l\geq 0} c_e(k,l;s) x^{2k+2l}= \sum_{k,l\geq 0}c'_e(k,l;s\,|\,a, c\,|\,q) x^{2k+2l},
\end{gather*}
where
\begin{gather}
c'_e(k,l;s\,|\,a,c\,|\,q) = {\big(q a^2/c^2;q^2\big)_k \big(q^3 s/c^2;q^2\big)_k \big(q^2 s^2/c^4;q^2\big)_k \over
\big(q^2;q^2\big)_k\big(q s/c^2;q^2\big)_k\big(q^3 s^2/a^2 c^2;q^2\big)_k}\big(q^2/a^2\big)^k \nonumber \\
\hphantom{c'_e(k,l;s\,|\,a,c\,|\,q) =}{} \times {\big(c^2/q^2;q\big)_l (s/q;q)_{2k+l} \over (q;q)_l \big(q^2 s/c^2;q\big)_{2k+l} }
{1-q^{2k+2l-1}s\over 1-q^{-1}s} \big(q^2/c^2\big)^l. \label{type-4}
\end{gather}
\end{lem}

\subsection[Koornwinder polynomial with one column diagram $P_{(1^r)}(x\,|\,a,b,c,d\,|\,q,t)$]{Koornwinder polynomial with one column diagram $\boldsymbol{P_{(1^r)}(x\,|\,a,b,c,d\,|\,q,t)}$} \label{P(r)}

We move on to the proof of Theorem \ref{KoornwinderPoly_1^r}. Recall that $n$ is a positive integer, $x=(x_1,\ldots,x_n)$ is a sequence of variables, and $P_{(1^r)}(x\,|\,a,b,c,d\,|\,q,t)$ denotes the Koornwinder polynomial with one column diagram~$(1^r)$.

\begin{proof}[Proof of Theorem \ref{KoornwinderPoly_1^r}] We consider the following special case of Theorem \ref{thm:reproduction} above
\begin{gather*}
 x=(x_1,\ldots,x_n), \qquad n\in \bbZ_{>0},\qquad
y=(y_1), \qquad m=1,\\
\Psi(x;y)=y^{-n}\prod_{i=1}^n (1-y x_i)(1-y/x_i)=y^{-n} \sum_{r \geq 0} (-1)^r E_r(x) y^r, \\
 \lambda=(r),\qquad r\in \bbZ_{\geq 0} \quad \text{and} \quad r\leq n, \qquad
 s=(s_1)=t^{n-r+1}, \qquad  V(y)=y^{-1}\big(1-y^2\big), \\
 \widehat{f}(y;s)=y^{-r+1+n} \widehat{\Phi}(y;s)= y^{-r+1+n} \sum_{k,l,i,j\geq 0} \widehat{c}_e\big(k,l;t^{i+j} s\big)\widehat{c}_o(i,j; s) y^{2k+2l+i+j},
\end{gather*}
where
\begin{gather*}
 \widehat{c}_e(k,l;s)=c_e(k,l;s\,|\,t/a,t/c\,|\,t),\qquad \widehat{c}_o(i,j; s)=c_o(i,j;s\,|\,t/a,t/b,t/c,t/d\,|\,t).
\end{gather*}
Then calculating the constant term in $y$, we have
\begin{gather}
\big[ \Psi(x;y)V(y) \widehat{f}(y;s)\big]_{1,y} = \left[ \left(\sum_{r \geq 0} (-1)^r E_r(x) y^{-n+r}\right) y^{-1}\big(1-y^2\big)\right.\nonumber\\
\left.\qquad\quad{}\times \left(y^{-r+1+n} \sum_{k,l,i,j\geq 0} \widehat{c}_e\big(k,l;t^{i+j} s\big)\widehat{c}_o(i,j; s) y^{2k+2l+i+j} \right) \right]_{1,y} \nonumber \\
\qquad{} = (-1)^r \sum_{k,l,i,j\geq 0}(-1)^{-i-j}E_{r-2k-2l-i-j} (x)\widehat{c}\,'_e(k,l;t^{n-r+1+i+j})\widehat{c}_o\big(i,j; t^{n-r+1}\big) \nonumber \\
\qquad{} = (-1)^r P_{(1^r)}(x|a,b,c,d|q,t), \label{UE}
 \end{gather}
where $\widehat{c}\,'_e(k,l;s)= c'_e(k,l;s\,|\,t/a,t/c\,|\,t)$ (see (\ref{type-4}) above). This proves Theorem \ref{KoornwinderPoly_1^r}.
\end{proof}

\subsection[Koornwinder polynomial $P_{(1^r)}(x\,|\,a,-a,c,-c\,|\,q,t)$ with one column diagram]{Koornwinder polynomial $\boldsymbol{P_{(1^r)}(x\,|\,a,-a,c,-c\,|\,q,t)}$\\ with one column diagram}\label{Mac-C}
In view of (\ref{param-C}), we need $\Phi(x;s\,|\,a,b,c,d\,|\,q)$ written for the parameters
\begin{gather*}
(t/a, -t/a, t/c, -t/c; t, t/q).
\end{gather*}
Note that in this case we have the simplification of the coefficient as $\widehat{c}_o(m,n;s)=\delta_{m,0}\delta_{n,0}$. Hence we have a twofold summation formula for the~$\Phi(x;s)$. By Lemma~\ref{1-y^2phi}, we have
\begin{gather}
 \big(1-y^2\big) \widehat{\Phi}(y;s)= \big(1-y^2\big) \Phi(y;s\,|\,t/a, -t/a, t/c, -t/c\,|\,t)\label{PHI} \\
= \sum_{k, l \geq 0} {\big( tc^2/a^2 ; t^2 \big)_k \big(stc^2 ; t^2\big)_k \big(s^2 c^4/t^2 ; t^2\big)_k \over \big( t^2 ; t^2 \big)_k \big(sc^2/t ; t^2\big)_k \big(s^2 a^2c^2/t ; t^2\big)_k }
{ \big(1/c^2 ; t\big)_l (s/t ; t)_{2k+l} \over (t ; t)_l \big(sc^2 ; t\big)_{2k+l} } { 1-st^{2k+2l-1} \over 1-st^{-1} } a^{2k}c^{2l} y^{2k+2l}.\nonumber
\end{gather}
Write $s=t^{n-r+1}$ for simplicity. Plugging (\ref{PHI}) in (\ref{UE}), we have
\begin{gather}
 P_{(1^r)}(x\,|\,a,-a,c,-c\,|\,q,t)= \sum_{0 \leq 2k + 2l \leq r}E_{r-2k-2l}(x){ \big(1/c^2 ; t\big)_l (s/t ; t)_{2k+l} \over (t ; t)_l \big(sc^2 ; t\big)_{2k+l} }{1-st^{2k+2l-1} \over 1-st^{-1} } c^{2l} \nonumber \\
\hphantom{P_{(1^r)}(x\,|\,a,-a,c,-c\,|\,q,t)=}{} \times { \big(tc^2/a^2 ; t^2\big)_k \big(stc^2 ; t^2\big)_k \big(s^2 c^4 /t^2 ; t^2\big)_k \over \big(t^2 ; t^2\big)_k \big(sc^2/t ; t^2\big)_k \big(s^2 a^2c^2/t ; t^2\big)_k}a^{2k}. \label{PCE}
\end{gather}
This proves the formula in Corollary \ref{MacdonaldPoly_1^r}.

\section[Transition matrices $\mathcal{B}(s)$, $\widetilde{\mathcal{B}}(s)$ and Bressoud's matrix inversion]{Transition matrices $\boldsymbol{\mathcal{B}(s)}$, $\boldsymbol{\widetilde{\mathcal{B}}(s)}$\\ and Bressoud's matrix inversion}\label{MatrixInverse}
\subsection{Bressoud's matrix inversion}

\begin{thm}[{\cite[p.~1, Theorem]{B}, \cite[p.~5, Corollary]{L}}] Let $\mathcal{M}(u,v;x,y;q)$ be the infinite even lower-triangle matrix with nonzero entries given by
\begin{gather}
\mathcal{M}_{r,r-2i}(u,v;x,y;q)= y^i v^i { (x/y ; q)_i \over (q ; q)_i}
{ \big(u q^{r-2i} ; q\big)_{2i} \over (uxq^{r-i} ; q)_i \big(uyq^{r-2i+1} ; q\big)_i }, \label{Bre}
\end{gather}
for $r, i \in \mathbb{Z}_{\geq 0}$, $i \leq [{r \over 2}] $. The we have
\begin{gather}
\mathcal{M}(u,v;x,y;q)\mathcal{M}(u,v;y,z;q)=\mathcal{M}(u,v;x,z;q). \label{Bressoud}
\end{gather}
In particular, $\mathcal{M}(u,v;x,y;q)$ and $\mathcal{M}(u,v;y,x;q)$ are mutually inverse.
\end{thm}

\begin{dfn}Set
\begin{gather*}
d_r = { \big(t^2 v^{1/2} ; t\big)_r \over \big(u^{1/2} ; t\big)_r } \big(u^{1/4}/v^{3/4}\big)^r.
\end{gather*}
Let $\widetilde{\mathcal{M}}(u,v;x,y;t)$ denotes the conjugation of the matrix $\mathcal{M}\big(u,v;x,y;t^2\big)$ by the $(d_r)$ with entries
\begin{gather*}
 \widetilde{\mathcal{M}}_{r,r-2i}(u,v;x,y;t)=
 \mathcal{M}_{r,r-2i}\big(u,v;x,y;t^2\big) d_r/d_{r-2i} \\
\qquad {}= { \big(x/y ; t^2\big)_i \over \big(t^2 ; t^2\big)_i}
{ \big(v^{1/2} t^{r-2i+2} ; t\big)_{2i} \over \big(u^{1/2} t^{r-2i} ; t\big)_{2i} }
{ \big(u t^{2r-4i} ; t^2\big)_{2i} \over \big(uxt^{2r-2i} ; t^2\big)_i \big(uyt^{2r-4i+2} ; t^2\big)_i }\big(y u^{1/2}/v^{1/2}\big)^{i}.
\end{gather*}
\end{dfn}

\subsection[Transition matrices $\mathcal{B}(s)$ and $\widetilde{\mathcal{B}}(s)$]{Transition matrices $\boldsymbol{\mathcal{B}(s)}$ and $\boldsymbol{\widetilde{\mathcal{B}}(s)}$}
\begin{dfn}Let $\mathcal{B}(s)$ and $\widetilde{\mathcal{B}}(s)$ be even lower triangular matrices defined by
\begin{gather*}
\mathcal{B}(s) =
\widetilde{\mathcal{M}}\big(t^{2}/s^2 c^4, 1/s^2 t^4; c^2/t a^2, 1/t^2;t\big) \mathcal{M}\big(1/s, t; 1/c^2, 1;t\big) , \\
\widetilde{\mathcal{B}}(s)= \mathcal{M}\big(1/s, t; 1, 1/c^2;t\big) \widetilde{\mathcal{M}}\big(t^{2}/s^2c^4, 1/s^2t^4; 1/t^2, c^2/t a^2;t\big).
\end{gather*}
\end{dfn}

\begin{prp}\label{mutually-inverse} The $\mathcal{B}(s)$ and $\widetilde{\mathcal{B}}(s)$ are mutually inverse.
\end{prp}

\begin{proof}This follows from Bressoud's matrix inversion (\ref{Bressoud}).
\end{proof}

\begin{thm}\label{thm:g_r}We have
\begin{subequations}
\begin{gather}
P_{(1^r)}(x\,|\,a,-a,c,-c\,|\,q,t) =\sum_{k=0}^{\lfloor {r \over 2} \rfloor}\mathcal{B}_{r,r-2k}(t^n)E_{r-2k}(x), \label{P-E}\\
E_{r}(x)=\sum_{k=0}^{\lfloor {r \over 2} \rfloor}\widetilde{\mathcal{B}}_{r,r-2k}(t^n)P_{(1^{r-2k})}(x\,|\,a,-a,c,-c\,|\,q,t) . \label{E-P}
\end{gather}
\end{subequations}
Writing the coefficients explicitly, these read
\begin{subequations}
\begin{gather}
 P_{(1^r)}(x\,|\,a,-a,c,-c\,|\,q,t) =\sum_{k=0}^{\lfloor {r \over 2} \rfloor}\sum_{l=0}^{\lfloor {r -2k \over 2} \rfloor}
E_{r-2k-2l}(x) { \big( 1/c^2 ; t \big)_l \big( s t^{2k-1} ; t\big)_l \big( s t^{2k} ; t\big)_{2l} \over ( t ; t )_l \big( s c^2 t^{2k} ; t\big)_l \big( s t^{2k-1} ; t\big)_{2l} } c^{2l} \nonumber \\
\hphantom{P_{(1^r)}(x\,|\,a,-a,c,-c\,|\,q,t)=}{} \times
{ \big(t c^2/a^2 ; t^2\big)_k \big(s tc^2 ; t^2\big)_k \big(s^2 c^4 /t^2 ; t^2\big)_k (s ; t)_{2k} \over \big(t^2 ; t^2\big)_k \big(sc^2/t ; t^2\big)_k \big(s^2 a^2c^2/t ; t^2\big)_k \big(sc^2 ; t\big)_{2k} } a^{2k},\label{P-to-G}\\
E_{r}(x) = \sum_{l=0}^{\lfloor {r \over 2} \rfloor }\sum_{k=0}^{ \lfloor{r-2l \over 2} \rfloor }P_{(1^{r-2l-2k})}(x\,|\,a,-a,c,-c\,|\,q,t)
{ \big(c^2 ; t\big)_l \over (t ; t)_l }{ \big(s t^l ; t\big)_{l+2k} \over \big(st^{l-1}c^2 ; t\big)_{l+2k} } \nonumber\\
\hphantom{E_{r}(x) =}{} \times { \big(a^2/tc^2 ; t^2\big)_k \over \big(t^2 ; t^2\big)_k } { \big(s^2 t^{4l-2} c^4 ; t^2\big)_k \over \big(s^2 t^{4l} c^4 ; t^2\big)_k }
{ \big(s^2 t^{4l+2k-2} c^4 ; t^2\big)_k \over \big(s^2 t^{4l+2k-3} a^2c^2 ; t^2\big)_k }\big( t c^2\big)^k, \label{e-p}
\end{gather}
\end{subequations}
where $s=t^{n-r+1}$. In particular form \eqref{e-p}, we have Theorem~{\rm \ref{THME-P}}.
\end{thm}

\begin{proof} Clearly, (\ref{PCE}) and (\ref{P-to-G}) are the same. We show that~(\ref{P-to-G}) and~(\ref{P-E}) are the same. By~(\ref{Bre}) and $s=t^{n-r+1}$, we have
\begin{gather*}
 { \big( 1/c^2 ; t \big)_l \big( s t^{2k-1} ; t\big)_l \big( s t^{2k} ; t\big)_{2l} \over ( t ; t )_l \big( s c^2 t^{2k} ; t\big)_l \big( s t^{2k-1} ; t\big)_{2l} } c^{2l}=
{ \big(1/c^2 ; t\big)_l \over (t ; t)_l} { \big( s^{-1} t^{-2k-2l+1} ; t\big)_{2l} \over \big(s^{-1} t^{-2k-l+1}/c^2 ; t\big)_l \big(s^{-1} t^{-2k-2l+2} ; t\big)_l } t^l \\
 \hphantom{{ \big( 1/c^2 ; t \big)_l \big( s t^{2k-1} ; t\big)_l \big( s t^{2k} ; t\big)_{2l} \over ( t ; t )_l \big( s c^2 t^{2k} ; t\big)_l \big( s t^{2k-1} ; t\big)_{2l} } c^{2l}}{} =
\mathcal{M}_{r-2k,r-2k-2l}\big(t^{-n}, t, 1/c^2, 1;t\big),
\end{gather*}
and
\begin{gather*}
 { \big(t c^2/a^2 ; t^2\big)_k \big(stc^2 ; t^2\big)_k \big(s^2 c^4 /t^2 ; t^2\big)_k (s ; t)_{2k} \over \big(t^2 ; t^2\big)_k \big(sc^2/t ; t^2\big)_k \big(s^2 a^2 c^2/t ; t^2\big)_k \big(sc^2 ; t\big)_{2k} } a^{2k} \\
 \qquad{} = { \big(tc^2/a^2 ; t^2\big)_k \over \big(t^2 ; t^2\big)_k } { \big(s^{-1} t^{-2k+1} ; t\big)_{2k} \over \big(s^{-1} t^{-2k+2}/c^2 ; t\big)_{2k} }
{ \big(s^{-2} t^{-4k+4}/c^4 ; t^2\big)_{2k} \over \big(s^{-2} t^{-2k+3}/a^2 c^2 ; t^2\big)_k \big(s^{-2} t^{-4k+4}/c^4 ; t^2\big)_k } \big(t/c^2\big)^k \\
 \qquad{} = \widetilde{\mathcal{M}}_{r, r-2k}\big(t^{-2n+2}/c^4, t^{-2n-4}, c^2/ta^2, 1/t^2;t\big).
\end{gather*}
As for (\ref{E-P}) and (\ref{e-p}), we have
\begin{gather*}
\widetilde{\mathcal{M}}_{r, r-2k-2l}\big(t^{-2n+2}/c^4, t^{-2n-4}, 1/t^2,c^2/ta^2;t\big)
\mathcal{M}_{r-2k,r-2k-2l}\big(t^{-n}, t,1, 1/c^2;t\big) \\
\qquad{} = { \big(a^2/t c^2 ; t^2\big)_k \over \big(t^2 ; t^2\big)_k }
{ \big(s^{-1} t^{-2l-2k+1} ; t\big)_{2k} \over \big(s^{-1} t^{-2l-2k+2}/c^2 ; t\big)_{2k} }
{ \big(s^{-2} t^{-4l-4k+4}/c^4 ; t^2\big)_{2k} \over \big(s^{-2} t^{-4l-2k+2}/c^4 ; t^2\big)_k \big(s^{-2} t^{-4l-4k+5}/c^2a^2 ; t^2\big)_k } \\
\qquad\quad{}\times \big(t^2/a^2\big)^k { \big(c^2 ; t\big)_l \over (t ; t)_l }
{\big(s^{-1} t^{-2l+1} ; t\big)_l \over {\big(s^{-1} t^{-2l+2}/c^2 ; t\big)_l} }\big(t/c^2\big)^l \\
\qquad{} =
{ \big(c^2 ; t\big)_l \over (t ; t)_l }{ \big(s t^l ; t\big)_{l+2k} \over \big(st^{l-1}c^2 ; t\big)_{l+2k} }
{ \big(a^2/tc^2 ; t^2\big)_k \over \big(t^2 ; t^2\big)_k }
{ \big(s^2 t^{4l-2} c^4 ; t^2\big)_k \over \big(s^2 t^{4l} c^4 ; t^2\big)_k }
{ \big(s^2 t^{4l+2k-2}c^4 ; t^2\big)_k \over \big(s^2 t^{4l+2k-3} a^2 c^2 ; t^2\big)_k }\big( tc^2\big)^k.\!\!\!\!\tag*{\qed}
\end{gather*}\renewcommand{\qed}{}
\end{proof}

\subsection[Entries of $\mathcal{B}(s)$ and $\widetilde{\mathcal{B}}(s)$ in terms of ${}_4\phi_3$ series]{Entries of $\boldsymbol{\mathcal{B}(s)}$ and $\boldsymbol{\widetilde{\mathcal{B}}(s)}$ in terms of $\boldsymbol{{}_4\phi_3}$ series}

Recall that we have defined $B(s,j)$ and $\widetilde{B}(s,j)$ in Definition~\ref{B-Bt-def} as
\begin{gather*}
 B(s,j)=(-1)^j s^{-j} {\big(s^2/t^2; t^2\big)_j \over \big(t^2; t^2\big)_j} {1-s^2 t^{4j-2} \over 1-s^2 t^{-2}}
{}_{4}\phi_3 \left[ { -sa^2, -sc^2, s^2 t^{2j-2}, t^{-2j}
\atop
 -s, -st, s^2a^2 c^2/t } ; t^2, t^2\right],\nonumber\\
 \widetilde{B}(s,j)=(st^{j-1})^{-j}{\big(t^{2j}s^2;t^2\big)_j\over \big(t^2;t^2\big)_j}
 {}_{4}\phi_3 \left[ {-t^{-2j+2}/sa^2, -t^{-2j+2}/sc^2, t^{-2j+2}/s^2 , t^{-2j} \atop -t^{-2j+1}/s, -t^{-2j+2}/s,t^{-4j+5}/s^2a^2c^2} ; t^2, t^2\right].
\end{gather*}

\begin{prp}We have
\begin{subequations}
\begin{gather}
B(s,j) =(-1)^j t^{j}s^{-j} {\big(s^2/t^2; t^2\big)_j \over \big(t^2; t^2\big)_j} {1-s t^{2j-1} \over 1-s t^{-1}} {}_{4}\phi_3 \left[ { -sa^2/t, -sc^2/t, s^2 t^{2j-2}, t^{-2j}
\atop -s, -s/t, s^2a^2c^2/t } ; t^2, t^2\right], \!\!\!\!\label{B[s,j]} \\
\widetilde{B}(s,j) =t^j \big(st^{j-1}\big)^{-j}{\big(s^2t^{2j};t^2\big)_j\over \big(t^2;t^2\big)_j}{ 1+st^{-1} \over 1+st^{2j-1}} \nonumber \\
\hphantom{\widetilde{B}(s,j) =}{} \times {}_{4}\phi_3 \left[ {- t^{-2j+3}/sa^2 ,-t^{-2j+3}/sc^2, t^{-2j + 2}/s^2, t^{-2j} \atop
 -t^{-2j+2}/s, - t^{-2j+3}/s, t^{-4j+5}/s^2a^2 c^2} ;t^2, t^2 \right]. \label{tildeB[s,j]}
\end{gather}
\end{subequations}
\end{prp}
\begin{proof}This follows from the Sears transformation \cite[p.~49, equation~(2.10.4)]{GR}.
\end{proof}

\begin{thm}\label{BIBASIC}We have
\begin{subequations}
\begin{gather}
\mathcal{B}_{r,r-2i}(s)=B\big(st^{-r+1},i\big), \label{Btophi-1}\\
\widetilde{\mathcal{B}}_{r,r-2i}(s)=\widetilde{B}\big(st^{-r+1},i\big). \label{Btophi-2}
\end{gather}
\end{subequations}
\end{thm}

\begin{proof} Recall that the bibasic hypergeometric series $\Phi$ (see \cite[p.~99, equation~(3.9.1)]{GR}) is defined by
\begin{gather*}
\Phi\left[{a_1,\ldots,a_{r+1}:c_1,\ldots,c_s \atop b_1,\ldots,b_r:d_1,\ldots,d_s};q,p;z \right]
=\sum_{n=0}^\infty {(a_1,\ldots,a_{r+1};q)_n\over (q,b_1,\ldots,b_r;q)_n} {(c_1,\ldots,c_s;p)_n\over (d_1,\ldots,d_s;p)_n} z^n.
\end{gather*}
We use the $q$-analogue of Bailey's transformation \cite[p.~99, equation~(3.10.14)]{GR}:
\begin{gather}
\Phi\left[
{a^2,at^2,-at^2,b^2,c^2:-at/w,t^{-i}\atop a,-a,a^2t^2/b^2,a^2t^2/c^2:w,-a t^{i+1}};t^2,t\,;{awt^{i+1}\over b^2 c^2}\right]\nonumber\\
\qquad{} ={(-at, at^2/w, w/at ; t)_i\over (-t, at/w, w;t)_i}
{}_5\phi_4\left[{at,at^2,a^2t^2/b^2c^2,a^2t^2/w^2,t^{-2i}\atop a^2t^2/b^2,a^2t^2/c^2,at^{2-i}/w,at^{3-i}/w};t^2,t^2\right].\label{Bailey1}
\end{gather}

When $a=c^2$, (\ref{Bailey1}) becomes
\begin{gather}
\Phi\left[{a^2,-at^2,b^2:-at/w,t^{-i}\atop -a,a^2t^2/b^2:w,-a t^{i+1}};t^2,t ;{wt^{i+1}\over b^2}\right]\nonumber \\
\qquad{} = {(-at, at^2/w, w/at ; t)_i\over (-t, at/w, w;t)_i} {}_4\phi_3\left[ {at,at^2/b^2,a^2t^2/w^2,t^{-2i} \atop a^2t^2/b^2,at^{2-i}/w,at^{3-i}/w};t^2,t^2 \right]. \label{Bailey1-1}
\end{gather}
Replacing the parameters in (\ref{Bailey1-1}) as
\begin{gather*}
\big(a, w, b^2\big) \rightarrow \big({-}sc^2/t, c^2 t^{-i+1}, tc^2/a^2\big),
\end{gather*}
we can prove (\ref{Btophi-1}) as
\begin{gather*}
 \mathcal{B}_{r,r-2i}\big(s t^{r-1}\big)= \sum_{k=0}^{i}{\big(1/c^2;t\big)_{i-k}\over (t;t)_{i-k}}{\big(st^{2k-1};t\big)_{i-k}\over \big(sc^2t^{2k};t\big)_{i-k}} {\big(st^{2k};t\big)_{2i-2k}\over \big(st^{2k-1};t\big)_{2i-2k}}c^{2i-2k} \\
\hphantom{\mathcal{B}_{r,r-2i}\big(s t^{r-1}\big)=}{}\times
{ \big(tc^2/a^2 ; t^2\big)_k \big(stc^2 ; t^2\big)_k \big(s^2 c^4 /t^2 ; t^2\big)_k (s ; t)_{2k} \over \big(t^2 ; t^2\big)_k \big(sc^2/t ; t^2\big)_k \big(s^2 a^2c^2/t ; t^2\big)_k \big(sc^2 ; t\big)_{2k} } a^{2k}\\
\hphantom{\mathcal{B}_{r,r-2i}\big(s t^{r-1}\big)}{} ={\big(1/c^2;t\big)_{i}\over (t;t)_{i}}
{(s/t;t)_{i}\over \big(sc^2;t\big)_{i}} {(s;t)_{2i}\over (s/t;t)_{2i}}c^{2i} \\
\hphantom{\mathcal{B}_{r,r-2i}\big(s t^{r-1}\big)=}{} \times \sum_{k=0}^{i}{\big(t^{-i};t\big)_k\over \big(c^2t^{-i+1};t\big)_k}
{\big(s t^{i-1};t\big)_k\over \big(s c^2 t^i;t\big)_k} { \big(tc^2/a^2 ; t^2\big)_k \big(stc^2 ; t^2\big)_k \big(s^2 c^4 /t^2 ; t^2\big)_k \over
\big(t^2 ; t^2\big)_k \big(sc^2/t ; t^2\big)_k \big(s^2 a^2 c^2/t ; t^2\big)_k } a^{2k} t^k \\
\hphantom{\mathcal{B}_{r,r-2i}\big(s t^{r-1}\big)}{}
= {\big(1/c^2;t\big)_{i}\over (t;t)_{i}} {(s/t;t)_{i}\over \big(sc^2;t\big)_{i}} {(s;t)_{2i}\over (s/t;t)_{2i}}c^{2i}
\Phi\left[{tc^2/a^2,sc^2t,s^2c^4/t^2:st^{i-1},t^{-i}\atop sc^2/t,s^2a^2 c^2/t:c^2t^{-i+1},sc^2t^i};t^2,t;a^2t\right]\\
\hphantom{\mathcal{B}_{r,r-2i}\big(s t^{r-1}\big)}{}
={\big(1/c^2;t\big)_{i}\over (t;t)_{i}} {(s/t;t)_{i}\over \big(sc^2;t\big)_{i}} {(s;t)_{2i}\over (s/t;t)_{2i}}c^{2i}\\
\hphantom{\mathcal{B}_{r,r-2i}\big(s t^{r-1}\big)=}{} \times
\big({-}1/sc^2\big)^i {\big(s c^2,-s/t;t\big)_i (-s;t)_{2i}\over \big(1/c^2,-t;t\big)_i (-s/t;t)_{2i}}
{}_4\phi_3\left[{-sa^2,-sc^2,s^2t^{2i-2},t^{-2i} \atop -s,-st,s^2a^2c^2/t};t^2,t^2 \right] \\
\hphantom{\mathcal{B}_{r,r-2i}\big(s t^{r-1}\big)}{}
=B(s,i).
\end{gather*}

When $at=c^2$, (\ref{Bailey1}) becomes
\begin{gather}
\Phi\left[{a^2,at^2,-at^2,b^2:-at/w,t^{-i}\atop a,-a,a^2t^2/b^2:w,-a t^{i+1}};t^2,t;{wt^{i}\over b^2}\right] \nonumber\\
\qquad{} = {(-at, at^2/w, w/at ; t)_i\over (-t, at/w, w;t)_i}{}_4\phi_3\left[{at^2,at/b^2,a^2t^2/w^2,t^{-2i}\atop a^2t^2/b^2,at^{2-i}/w,at^{3-i}/w};t^2,t^2\right].\label{Bailey1-2}
\end{gather}
Replacing the parameters in (\ref{Bailey1-2}) as
\begin{gather*}
\big(a, w, b^2\big) \rightarrow \big({-}t^{-2i+1}/sc^2, t^{-i+1}/c^2, a^2/tc^2\big),
\end{gather*}
we can prove (\ref{Btophi-2}) as
\begin{gather*}
 \widetilde{\mathcal{B}}_{r,r-2i}\big(st^{r-1}\big) = \sum_{k=0}^{i}
\dfrac{\big(c^2, t^{-2i+2k+1} /s; t\big)_{i-k} }{\big(t, t^{-2i+2k+2}/sc^2 ; t\big)_{i-k}} \big(t/c^2\big)^{i-k} {\big(a^2/tc^2 ; t^2\big)_{k} \over \big(t^2 ; t^2\big)_k}
{\big(t^{-2i+1} /s; t\big)_{2k} \over \big(t^{-2i+2}/sc^2 ; t\big)_{2k}} \\
\qquad\quad{} \times
{ \big(t^{-4i+4}/s^2c^4 ; t^2\big)_{2k} \over \big(t^{-4i+2k+2}/s^2c^4 ; t^2\big)_{k} \big(t^{-4i+5}/s^2a^2 c^2 ; t^2\big)_{k}}
\big(t^2/a^2\big)^k \nonumber \\
\qquad{} = \dfrac{\big(c^2, t^{-2i+1}/s ; t\big)_i }{\big(t, t^{-2i+2}/sc^2 ; t\big)_i } \big(t/c^2\big)^i \\
\qquad\quad{} \times \Phi\left[{a^2/tc^2, t^{-4i+2}/s^2c^4, t^{-2i+3}/sc^2, - t^{-2i+3}/sc^2 : t^{-i+1}/s, t^{-i}
\atop t^{-4i+5}/s^2a^2c^2, t^{-2i+1}/sc^2, - t^{-2i+1}/sc^2 : t^{-i+2}/sc^2, t^{-i+1}/c^2}
; t^2 , t ; t^2/a^2 \right] \\
\qquad{} = \dfrac{(- s /t, s ; t)_{2i} } {\big(t^2, s^2 ; t^2\big)_i } ( t/s)^i t^{-i(i-1)} {}_{4}\phi_3 \left[ {- t^{-2i+3}/sc^2 ,- t^{-2i+3}/sa^2, t^{-2i + 2}/s^2, t^{-2i}\atop t^{-4i+5}/s^2a^2 c^2, - t^{-2i+2}/s, - t^{-2i+3}/s}
;t^2, t^2 \right]\nonumber\\
\qquad{} =\widetilde{B}(s ,i).\tag*{\qed}
\end{gather*}\renewcommand{\qed}{}
\end{proof}

\section[Four term relations for $B(s,i)$ and $\widetilde{B}(s,i)$]{Four term relations for $\boldsymbol{B(s,i)}$ and $\boldsymbol{\widetilde{B}(s,i)}$}\label{BBt}

\subsection{Four term relations}
Recall that we have defined $f(s)$ in (\ref{f[s]}) and have introduced the notation $F(s,l)$ in (\ref{F[s,j]}) as
\begin{gather*}
F(s,l)=f\big(s/t^l\big)=\dfrac{\big(1-t^l/s \big)\big(1-t^{l+2}/s a^2c^2 \big)\big(1+t^{l+1}/s a^2\big)\big(1+t^{l+1}/sc^2\big)} {\big(1-t^{2l+1}/s^{2}a^2 c^2\big)\big(1-t^{2l+3}/s^{2}a^2 c^2\big)}.
\end{gather*}

\begin{thm}\label{four-term}We have
\begin{subequations}
\begin{gather}
B(s,i)+F(s,-1)B\big(st^{2},i-1\big) =B(st,i) + B(st,i-1),\label{B4-1}\\
\widetilde{B}(s,i) +F(s,2-2i) \widetilde{B}(s,i-1) =\widetilde{B}\big(st^{-1},i\big)+\widetilde{B}(st,i-1).\label{B4-2}
\end{gather}
\end{subequations}
\end{thm}

\begin{proof}Recall the shorthand notation (\ref{b(s,i)}), (\ref{4phi3(s,i,k)}), and $B(s,i) = b(s,i) \!\sum\limits_{k=0}^i\! {}_{4}\phi_3(s,i,k)$. For~(\ref{B4-1}), we need to show
\begin{gather}
\sum_{k=0}^{i} {}_{4}\phi_3(s, i, k)+F(s,-1) \dfrac{b(st^2,i-1)}{b(s,i)}\sum_{k=0}^{i-1} {}_{4}\phi_3\big(st^2, i-1, k\big)\nonumber \\
\qquad{} = \dfrac{b(st,i)}{b(s,i)} \sum_{k=0}^{i} {}_{4}\phi_3(st, i, k)+ \dfrac{b(st, i-1)}{b(s, i)} \sum_{k=0}^{i-1} {}_{4}\phi_3(st, i-1, k).\label{MainProp0}
\end{gather}

Firstly, we have
\begin{gather}
\text{l.h.s.\ of } (\ref{MainProp0}) = 1 + \sum_{k=1}^{i} \left( {}_{4}\phi_3(s, i, k) + F(s,-1) \dfrac{b(s,i-1)}{b(s,i)} {}_{4}\phi_3\big(st^2, i-1, k-1\big) \right)\nonumber \\
\hphantom{\text{l.h.s.\ of } (\ref{MainProp0})}{} = 1 + \sum_{k=1}^{i}\dfrac{ \big({-}sa^2, -sc^2, st^2, s^2 t^{2i-2}, t^{-2i} ; t^2 \big)_k }{ \big( t^2, s, -s, -st, s^2a^2 c^2t ; t^2 \big)_k} t^{2k}\nonumber \\
\hphantom{\text{l.h.s.\ of } (\ref{MainProp0})}{} = {}_{5}\phi_4 \left[ {-sa^2, -sc^2, st^2, s^2 t^{2i-2}, t^{-2i}\atop s, -s, -st, s^2a^2 c^2t } ; t^2 , t^2 \right]. \label{mp_lhs0}
\end{gather}

Nextly, by setting
\begin{gather*}
\overline{{}_{4}\phi_3}(s,i,k):= { \big( t^{-2i}, s^2 t^{2i-2}, -sc^2/t, -sa^2/t ; t^2 \big)_k \over \big( t^2, s^2a^2 c^2/t, -s, -s/t ;t^2\big)_k }t^{2k},
\end{gather*}
we can write
\begin{gather*}
\sum_{k=0}^i {}_{4}\phi_3(s,i,k) = {1+s^{-1}t \over 1+s^{-1}t^{-2i+1} } t^{-i} \sum_{k=0}^i \overline{{}_{4}\phi_3}(s,i,k).
\end{gather*}
Hence we have
\begin{gather*}
\text{r.h.s.\ of } (\ref{MainProp0}) =
\sum_{k=0}^{i} \left(\dfrac{b(st,i)}{b(s,i)}{ 1+s^{-1} \over 1+s^{-1}t^{-2i} }t^{-i} \,\overline{{}_{4}\phi_3}(st, i, k) \right. \\
\left. \hphantom{\text{r.h.s.\ of } (\ref{MainProp0}) =}{}
+ \dfrac{b(st,i-1)}{b(s,i)} { 1+s^{-1} \over 1+s^{-1}t^{-2i+2} }t^{-i+1} \,\overline{{}_{4}\phi_3}(st, i-1, k)\right)\\
\hphantom{\text{r.h.s.\ of } (\ref{MainProp0}) =}{} =\sum_{k=0}^{i} \dfrac{ \big( {-}sa^2, -sc^2, st^2, s^2 t^{2i-2}, t^{-2i} ; t^2 \big)_k }{ \big( t^2, s, -s, -st, s^2a^2 c^2t ; t^2 \big)_k} t^{2k}\\
\hphantom{\text{r.h.s.\ of } (\ref{MainProp0}) =}{} = {}_{5}\phi_4 \left[ {-sa^2, -sc^2, st^2, s^2 t^{2i-2}, t^{-2i}\atop s, -s, -st, s^2a^2 c^2t } ; t^2 , t^2 \right]
 =\text{l.h.s.\ of } (\ref{mp_lhs0}).
\end{gather*}

Now we turn to (\ref{B4-2}). Set
\begin{subequations}
\begin{gather}
\widetilde{b}(s, i) :=\dfrac{\big({-} s t^{-1}, s ; t\big)_{2i} } {\big(t^2, s^2 ; t^2\big)_i } \big(s^{-1} t\big)^i t^{-i(i-1)}, \label{tilde-b} \\
{}_{4}{\widetilde{\phi}}_3(s, i, k) := { \big({-}s^{-1} t^{-2i+3}/c^2 ,- s^{-1} t^{-2i+3}/a^2, s^{-2} t^{-2i + 2}, t^{-2i} ; t^2 \big)_i \over
 \big( s^{-2} t^{-4i+5}/a^2 c^2, -s^{-1} t^{-2i+2}, - s^{-1} t^{-2i+3} ; t^2 \big)_i }t^{2k},\label{tilde-phi}\\
\overline{{}_{4}\widetilde{\phi}_3}(s, i, k) :=\dfrac{ \big(t^{-2i}, s^{-2} t^{-2i + 2}, -s^{-1}t^{-2i+2}/a^2 ,-s^{-1}t^{-2i+2}/a^2; t^2 \big)_k }
{ \big( s^{-2} t^{-4i+5}/a^2 c^2, -s^{-1}t^{-2i+2}, -s^{-1}t^{-2i+1}; t^2 \big)_k }t^{2k},
\end{gather}
\end{subequations}
for simplicity. Then we can write
\begin{gather*}
\widetilde{B}(s, i) =\widetilde{b}(s, i) \sum_{k=0}^{i} {}_{4}{\widetilde{\phi}}_3(s, i, k),\\
\sum_{k=0}^{i} {}_{4}{\widetilde{\phi}}_3(s, i, k) = { 1+s^{-1}t^{-2i+1} \over 1+s^{-1}t }t^{i} \sum_{k=0}^{i} \overline{{}_{4}\widetilde{\phi}_3}(s, i, k).
\end{gather*}

We shall show
\begin{gather}
 \sum_{k=0}^{i} {}_{4}{\widetilde{\phi}}_3(s, i, k)+F(s,2-2i) \dfrac{\widetilde{b}(s,i-1)}{\widetilde{b}(s,i)} \sum_{k=0}^{i-1} {}_{4}{\widetilde{\phi}}_3(s, i-1, k)\nonumber \\
\qquad {}= \dfrac{\widetilde{b}(st^{-1},i)}{\widetilde{b}(s,i)}\sum_{k=0}^{i} {}_{4}{\widetilde{\phi}}_3\big(st^{-1}, i, k\big)+\dfrac{\widetilde{b}(st, i-1)}{\widetilde{b}(s, i)}
\sum_{k=0}^{i-1} {}_{4}{\widetilde{\phi}}_3(st, i-1, k).\label{MainProp}
\end{gather}
We have
\begin{gather}
\text{l.h.s.\ of } (\ref{MainProp}) =
1 + \sum_{k=1}^{i} \left( {}_{4}{\widetilde{\phi}}_3(s, i, k) + F(s,2-2i) \dfrac{\widetilde{b}(s,i-1)}{\widetilde{b}(s,i)} {}_{4}{\widetilde{\phi}}_3(s, i-1, k-1) \right)\nonumber \\
\hphantom{\text{l.h.s.\ of } (\ref{MainProp})}{} =
 1 + \sum_{k=1}^{i} \dfrac{ \big( t^{-2i}, s^{-2}t^{-2i+2}, -s^{-1}t^{-2i+3}/a^2, -s^{-1}t^{-2i+3}/c^2, s^{-1}t^{-2i+3}; t^2 \big)_k }
{ \big( t^2, s^{-2}t^{-4i+7}/a^2 c^2, -s^{-1}t^{-2i+3}, -s^{-1}t^{-2i+2}, s^{-1}t^{-2i+1}; t^2 \big)_k} t^{2k}\nonumber \\
\hphantom{\text{l.h.s.\ of } (\ref{MainProp})}{} =
 {}_{5}\phi_4 \left[ { t^{-2i}, s^{-2}t^{-2i+2}, -s^{-1}t^{-2i+3}/a^2, -s^{-1}t^{-2i+3}/c^2, s^{-1}t^{-2i+3}
\atop s^{-2}t^{-4i+7}/a^2 c^2, -s^{-1}t^{-2i+3}, -s^{-1}t^{-2i+2}, s^{-1}t^{-2i+1}} ; t^2 , t^2 \right]\!. \!\!\!\!\!\label{mp_lhs}
\end{gather}
On the other hand, we have
\begin{gather*}
\text{r.h.s.\ of } (\ref{MainProp}) =\sum_{k=0}^{i} \left(
\dfrac{\widetilde{b}(st^{-1},i)}{\widetilde{b}(s,i)}{ 1+s^{-1}t^{-2i+2} \over 1+s^{-1}t^{2} } t^{i} \,\,
\overline{{}_{4}\widetilde{\phi}_3}\big(st^{-1}, i, k\big)\right.\\
\left. \hphantom{\text{r.h.s.\ of } (\ref{MainProp})=}{}+
\dfrac{\widetilde{b}(st, i-1)}{\widetilde{b}(s, i)} { 1+s^{-1}t^{-2i+2} \over 1+s^{-1} } t^{i-1} \,\,
 \overline{{}_{4}\widetilde{\phi}_3}(st, i-1, k) \right) \\
\hphantom{\text{r.h.s.\ of } (\ref{MainProp})}{} =
 \sum_{k=1}^{i} \dfrac{ \big( t^{-2i}, s^{-2}t^{-2i+2}, -s^{-1}t^{-2i+3}/a^2, -s^{-1}t^{-2i+3}/c^2, s^{-1}t^{-2i+3}; t^2 \big)_k }
{ \big( t^2, s^{-2}t^{-4i+7}/a^2 c^2, -s^{-1}t^{-2i+3}, -s^{-1}t^{-2i+2}, s^{-1}t^{-2i+1}; t^2 \big)_k} t^{2k}\\
\hphantom{\text{r.h.s.\ of } (\ref{MainProp})}{} =
 {}_{5}\phi_4 \left[ { t^{-2i}, s^{-2}t^{-2i+2}, -s^{-1}t^{-2i+3}/a^2, -s^{-1}t^{-2i+3}/c^2, s^{-1}t^{-2i+3}
\atop
 s^{-2}t^{-4i+7}/a^2 c^2, -s^{-1}t^{-2i+3}, -s^{-1}t^{-2i+2}, s^{-1}t^{-2i+1}} ; t^2 , t^2 \right] \\
\hphantom{\text{r.h.s.\ of } (\ref{MainProp})}{}
 = \text{r.h.s.\ of } (\ref{mp_lhs}). \tag*{\qed}
\end{gather*}\renewcommand{\qed}{}
\end{proof}

\subsection{Another proof of Theorem \ref{mutually-inverse}}
As an application of the four term relations $B(s,i)$ and $\widetilde{B}(s,i)$, we present another proof of Theorem~\ref{mutually-inverse}, providing an amusing complementary argument based on the Bressoud matrix inversion.
\begin{prp}\label{ANOTHER}The four terms relations in Theorem~{\rm \ref{four-term}} imply that
\begin{gather*}
\sum_{k=0}^i B(s,k)\widetilde{B}\big(s t^{2k},i-k\big)=\delta_{i,0},\qquad
\sum_{k=0}^i \widetilde{B}(s,k)B\big(s t^{2k},i-k\big)=\delta_{i,0},
\end{gather*}
hence, that the matrices $\mathcal{B}(s)$ and $\widetilde{\mathcal{B}}(s)$ are mutually inverse.
\end{prp}

\begin{proof} Set
\begin{gather*}
{\rm l.h.s.}(s,i):=\sum_{k=0}^i B(s,k)\widetilde{B}\big(s t^{2k},i-k\big),
\end{gather*}
for simplicity.

First we show that for $i\geq 0$ we have the difference equation
\begin{gather}
{\rm l.h.s.}(s,i)-{\rm l.h.s.}(s/t,i)=0.\label{dif-eq}
\end{gather}

We prove this by induction. The case $i=0$ is clearly correct. Suppose that it is valid for $i-1$. Then we have
\begin{gather*}
{\rm l.h.s.}(s,i)-{\rm l.h.s.}(s/t,i) = \sum_{k=0}^i B(s,k)\big(\widetilde{B}\big(s t^{2k-1},i-k\big)\\
\qquad\quad{} -F(s,2-2i) \widetilde{B}\big(s t^{2k},i-k-1\big)+\widetilde{B}\big(s t^{2k+1},i-k-1\big) \big)\\
\qquad \quad {} - \sum_{k=0}^i \big(B(s,k)-F(s,0) B(s t,k-1)+B(s,k-1) \big)\widetilde{B}\big(s t^{2k-1},i-k\big)\\
\qquad{} = -F(s,2-2i) \,{\rm l.h.s.}(s,i-1)+F(s,0)\, {\rm l.h.s.}(s t,i-1)=0.
\end{gather*}

By definition ${\rm l.h.s.}(s,i)$ is a rational function in $s$, and it satisfies the difference equation~(\ref{dif-eq}). Therefore, ${\rm l.h.s.}(s,i)$ must be a constant. We have ${\rm l.h.s.}(s,0)=1$. Then we can check that for $i>0$ ${\rm l.h.s.}(1,i)=0$ (hence ${\rm l.h.s.}(s,i)=0$) by using the following lemma as
\begin{gather*}
 {\rm l.h.s.}(s,i)=\sum_{k=0}^i B(1,k)\widetilde{B}\big( t^{2k},i-k\big)= B(1,0)\widetilde{B}( 1,i)-B(1,1). \tag*{\qed}
\end{gather*}\renewcommand{\qed}{}
\end{proof}

\begin{lem}\label{special-B} We have
\begin{gather}
B(1,j)=
\begin{cases}
\hphantom{-}1,& j=0,\\
-1,& j=1,\\
\hphantom{-}0,& j>1,
\end{cases}\label{B1j}\\
\widetilde{B}(1,i)-\widetilde{B}\big(t^2,i-1\big)=0,\qquad i>1.\label{Bt1j}
\end{gather}
\end{lem}

\begin{proof} The (\ref{B1j}) follows from the definition of $B(s,i)$. By noting
\begin{gather*}
 \widetilde{b}(s,i){}_4\widetilde{\phi}_3(s,i,k)- \widetilde{b}\big(st^2,i-1\big){}_4\widetilde{\phi}_3\big(st^2,i-1,k\big) =
 s^{-i} t^{-i(i-1)} {\big( t^{2i+2}s^2;t^2\big)_i \over \big(t^2;t^2\big)_i } {1-s\over 1-s t^{2i}} \\
 \qquad{}\times {\big({-}t^{-2i+2}/sa^2;t^2\big)_k\over \big({-}t^{-2i}/s;t^2\big)_k} {\big({-}t^{-2i+2}/sc^2;t^2\big)_k\over \big({-}t^{-2i+1}/s;t^2\big)_k}
{\big(t^{-2i}/s^2;t^2\big)_k\over \big(t^{-4i+5}/s^2a^2 c^2;t^2\big)_k} {\big(t^{-2i};t^2\big)_k\over \big(t^2;t^2\big)_k} t^{2k},
\end{gather*}
where we used the notation (\ref{tilde-b}) and (\ref{tilde-phi}), we have (\ref{Bt1j}) from the identity
\begin{gather*}
 \widetilde{B}(s,i)-\widetilde{B}\big(s t^2,i-1\big) = s^{-i} t^{-i(i-1)} {\big( t^{2i+2}s^2;t^2\big)_i \over \big(t^2;t^2\big)_i }
{1-s\over 1-s t^{2i}} \\
\qquad{} \times {}_4\phi_3 \left[
{ -t^{-2i+2}/sa^2,-t^{-2i+2}/sc^2, t^{-2i}/s^2,t^{-2i} \atop
-t^{-2i}/s,-t^{-2i+1}/s, t^{-4i+5}/s^2 a^2 c^2};t^2,t^2\right].\tag*{\qed}
\end{gather*}\renewcommand{\qed}{}
\end{proof}

\section[Transition matrix $\mathcal{C}$ and $(a,c,t)$-deformation of Catalan triangle three term recursion relations]{Transition matrix $\boldsymbol{\mathcal{C}}$ and $\boldsymbol{(a,c,t)}$-deformation\\ of Catalan triangle three term recursion relations}\label{SEC-C}

\subsection[Coefficient $C(s,j)$]{Coefficient $\boldsymbol{C(s,j)}$}

Recall that in Definition \ref{C(s,j)}, we have defined the function $C(s,j)$ as
\begin{align}
C(s,j):=\sum_{i=0}^{j} B(s,i) \binom{ m+2j }{ j-i }.\label{def-C}
\end{align}
Then (\ref{P-E}), (\ref{Btophi-1}), and (\ref{def-C}) imply (Theorem \ref{Pvsm})
\begin{gather*}
P_{(1^r)}(x\,|\,a,-a,c,-c\,|\,q,t) =\sum_{k=0}^{\lfloor {r \over 2} \rfloor} C\big(t^{n-r+1},k\big) m_{(1^{r-2k})}(x).
\end{gather*}

\subsection{Deformed Catalan triangle recursion relations}

\begin{prp}\label{THREE-1}We have the three term relation
\begin{gather}
 C(s,j) + F(s, -1) C\big(st^2,j-1\big)= C(st,j) . \label{C+FC=C}
\end{gather}
\end{prp}
\begin{proof} We have
\begin{gather*}
 C(s,j)+F(s,-1)C\big(st^2,j-1\big) \\
 \qquad{} =\sum_{i=0}^j B(s,i) \binom{ m+2j }{ j-i }+\sum_{i=0}^{j-1} F(s, -1) B\big(st^2 ,i\big) \binom{ m+2j }{ j-1-i } \\
\qquad{} =\binom{ m+2j }{ j }+\sum_{i=1}^{j} \left(B(st,i)+ B(st ,i-1) \right) \binom{ m+2j }{ j-i }\\
\qquad{} =\binom{ m+2j }{ j }+\binom{ m+2j }{ j-1 }+B(st,j)+\sum_{i=1}^{j-1} B(st,i)\left( \binom{ m+2j }{ j-i }+\binom{ m+2j }{ j-i-1 } \right) \\
\qquad{} =\sum_{i=0}^{j} B(st,i) \binom{ m+1+2j }{ j-i }=C(st,j) .\tag*{\qed}
\end{gather*}\renewcommand{\qed}{}
\end{proof}

\begin{prp}\label{THREE-2}We have
\begin{gather*}
 C(1,j)=\delta_{j,0}.
\end{gather*}
Hence the three term relation \eqref{C+FC=C} for $s=1$ reads
\begin{gather*}
 F(1, -1) C\big(t^2,j-1\big)= C(t,j).
\end{gather*}
\end{prp}
\begin{proof} We have $C(1,0)=1$. From Lemma~\ref{special-B}, we have for $j>0$
\begin{gather*}
 C(1,j)=B(1,0) \binom{ -1+2j }{ j }+B(1,1) \binom{ -1+2j }{ j-1} =\binom{ -1+2j }{ j }- \binom{ -1+2j }{ j-1 }=0. \tag*{\qed}
\end{gather*}\renewcommand{\qed}{}
\end{proof}

\subsection{Solution to the deformed Catalan triangle recursion relations}

\begin{thm} \label{FFF} We have $C(t^{r+1},0)=1$ for $r\in \mathbb{Z}_{\geq 0}$, and for $i\in \mathbb{Z}_{> 0},r\in \mathbb{Z}_{\geq 0}$ we have
\begin{gather}
 C\big(t^{r+1},i\big)= \sum_{(d_1,\ldots,d_i)\in \mathcal{P}[r,i]} F\big(t^{r+1}, d_1\big)F\big(t^{r+1}, d_2\big) \cdots F\big(t^{r+1}, d_i\big), \label{FFFF}
\end{gather}
where $\mathcal{P}[r,i]$ denotes the finite set defined by
\begin{gather*}
\mathcal{P}[r,i]=\big\{(d_1,d_2,\ldots,d_i)\in \mathbb{Z}^i\,|\, 0 \leq d_1\leq r ,\, d_{k} -1 \leq d_{k+1}\leq r \,\, {\rm for}\,\, 1\leq k<i \big\}.
\end{gather*}
\end{thm}

We prepare some lemmas.
\begin{lem}\label{lemma-a}For $r\in \mathbb{Z}_{\geq 0}$, we have
\begin{gather*}
C\big(t^{r+1},i+1\big) =\sum_{k=0}^{r}F\big(t^{k},-1\big) C\big(t^{k+2},i\big).
\end{gather*}
\end{lem}
\begin{proof}The case $r=0$ holds since $C(t,i+1)=F(1,-1) C\big(t^{2},i\big)$. Then we can show the induction step as
\begin{gather*}
 C\big(t^{r+2},i+1\big)=C\big(t^{r+1},i+1\big)+F\big(t^{r+1},-1\big) C\big(t^{r+3},i\big) \nonumber\\
\qquad{} =\sum_{k=0}^{r}F\big(t^{k},-1\big) C\big(t^{k+2},i\big)+F\big(t^{r+1},-1\big) C\big(t^{r+3},i\big)=\sum_{k=0}^{r+1}F\big(t^{k},-1\big) C\big(t^{k+2},i\big).\tag*{\qed}
\end{gather*}\renewcommand{\qed}{}
\end{proof}

\begin{lem}\label{lemma-b}We have
\begin{gather*}
 \mathcal{P}[r,i+1] =
\big\{ (d,d_1,d_2,\ldots,d_i)\in \mathbb{Z}^i\,|\\
\hphantom{\mathcal{P}[r,i+1] = \big\{}{} 0 \leq d_1\leq r ,\, (d_1-d+1,\ldots,d_i-d+1)\in \mathcal{P}[r-d+1,i+1]\big\}.
\end{gather*}
\end{lem}

\begin{proof}[Proof of Proposition \ref{FFF}.] We prove (\ref{FFFF}) by induction on $i$. It holds for $i=0$, since we have $C\big[t^{r+1},0\big]=1$, $r\in \mathbb{Z}_{\geq 0}$. The induction step is shown as follows. Lemmas~\ref{lemma-a} and~\ref{lemma-b} and the induction hypothesis give us
\begin{gather*}
C\big(t^{r+1},i+1\big) =\sum_{k=0}^{r}F\big(t^{k},-1\big) C\big(t^{k+2},i\big)=\sum_{d=0}^{r}F\big(t^{r+1},d\big) C\big(t^{r-d+2},i\big) \\
\hphantom{C\big(t^{r+1},i+1\big)}{} =\sum_{d=0}^{r}F\big(t^{r+1},d\big) \sum_{(d_1,\ldots,d_i)\in \mathcal{P}[r-d+1,i]} F\big(t^{r+1}, d_1\big)F\big(t^{r+1}, d_2\big) \cdots F\big(t^{r+1}, d_i\big) \\
\hphantom{C\big(t^{r+1},i+1\big)}{} = \sum_{(d,d_1,\ldots,d_i)\in \mathcal{P}[r,i+1]}F\big(t^{r+1}, d\big) F\big(t^{r+1}, d_1\big)F\big(t^{r+1}, d_2\big) \cdots F\big(t^{r+1}, d_i\big).\tag*{\qed}
\end{gather*}\renewcommand{\qed}{}
\end{proof}

\section[Some degenerations of Macdonald polynomials of types $C_n$ and $D_n$ with one column diagrams and Kostka polynomials]{Some degenerations of Macdonald polynomials of types $\boldsymbol{C_n}$\\ and $\boldsymbol{D_n}$ with one column diagrams and Kostka polynomials}\label{DEGEN}

This section is devoted to the study of several degenerations of our formulas for the Macdonald polynomial $P^{(C_n, C_n)}_{(1^r)}(x\,|\,b;q,t)$ (see Section~\ref{notation-C}).

\subsection[Some degenerations of $B(s,j)$ and $\widetilde{B}(s,j)$]{Some degenerations of $\boldsymbol{B(s,j)}$ and $\boldsymbol{\widetilde{B}(s,j)}$}

\subsubsection[$(C_n,C_n)$ case]{$\boldsymbol{(C_n,C_n)}$ case}
\begin{prp} When $a=t^{1/2}$, $c=q^{1/2}t^{1/2}$ in the equations \eqref{B[s,j]} and \eqref{tildeB[s,j]}, we have
%\begin{subequations}
\begin{gather*}
 B(s,j)= {\big(1/qt;t^2\big)_j \big(s^2/t^2,t^2\big)_j\over \big(t^2;t^2\big)_j\big(s^2 q t ;t^2\big)_j}{1-s^2 t^{4j-2}\over 1-s^2 t^{-2}}(q t)^j,\\
 \widetilde{B}(s,j)= {\big(qt;t^2\big)_j \big(s^2 t^{2j},t^2\big)_j\over \big(t^2;t^2\big)_j\big( s^2q t^{2j-1};t^2\big)_j}.
\end{gather*}
%\end{subequations}
\end{prp}

\begin{proof} Setting $a=t^{1/2}$, $c=q^{1/2}t^{1/2}$, we have by the Saalsch\"utz summation formula \cite[p.~17, equation~(1.7.2)]{GR} that
\begin{gather*}
B(s,j) =(-1)^j s^{-j} {\big(s^2/t^2; t^2\big)_j \over \big(t^2; t^2\big)_j}{1-s^2 t^{4j-2} \over 1-s^2 t^{-2}}{}_{3}\phi_2 \left[ { -sqt, s^2 t^{2j-2}, t^{-2j}\atop -s, s^2qt } ; t^2, t^2\right] \\
 \hphantom{B(s,j)}{} =
 (-1)^j s^{-j} {\big(s^2/t^2; t^2\big)_j \over \big(t^2; t^2\big)_j} {1-s^2 t^{4j-2} \over 1-s^2 t^{-2}} {\big(1/qt;t^2\big)_j\big({-}t^{-2j+2}/s;t^2\big)_j\over \big({-}s;t^2\big)_j \big(t^{-2j+2}/s^2qt;t^2\big)_j} \\
\hphantom{B(s,j)}{} = {\big(1/qt;t^2\big)_j \big(s^2/t^2,t^2\big)_j\over \big(t^2;t^2\big)_j\big(s^2 q t ;t^2\big)_j}{1-s^2 t^{4j-2}\over 1-s^2 t^{-2}}(q t)^j
\end{gather*}
and
\begin{gather*}
\widetilde{B}(s,j) =\big(st^{j-1}\big)^{-j}{\big(t^{2j}s^2;t^2\big)_j\over \big(t^2;t^2\big)_j} {}_{3}\phi_2 \left[ {-t^{-2j+1}/sq, t^{-2j+2}/s^2 , t^{-2j}\atop -t^{-2j+2}/s,t^{-4j+3}/s^2q} ; t^2, t^2\right]\\
\hphantom{\widetilde{B}(s,j)}{} = \big(st^{j-1}\big)^{-j} {\big(t^{2j}s^2;t^2\big)_j\over \big(t^2;t^2\big)_j} {\big(qt;t^2\big)_j \big({-}s;t^2\big)_j\over \big({-}t^{-2j+2}/s;t^2\big)_j \big(s^2 q t^{2j-1};t^2\big)_j}
 ={\big(qt;t^2\big)_j \big(s^2 t^{2j},t^2\big)_j\over \big(t^2;t^2\big)_j\big( s^2q t^{2j-1};t^2\big)_j}.\!\!\!\!\!\!\!\tag*{\qed}
\end{gather*}\renewcommand{\qed}{}
\end{proof}

\begin{cor} When $a=q^{1/2}$, $c=q$, $t=q$, we have
%\begin{subequations}
\begin{gather*}
 B(s,j)=
\begin{cases}
\hphantom{-}1,& j=0, \\
-1,& j=1, \\
\hphantom{-}0,& j>1,
\end{cases} \qquad
 \widetilde{B}(s,j)=1,\qquad j \geq 0.
\end{gather*}
%\end{subequations}
\end{cor}

\begin{cor}Let $m\in \mathbb{C}$. We have
%\begin{subequations}
\begin{gather*}
 \lim_{q \rightarrow 0} B\big(t^{m+1},j\big) \bigr|_{a=t^{1/2},\, c=q^{1/2}t^{1/2}}=(-1)^j t^{j(j-1)} { [m+2j]_{t^2} \over [m]_{t^2}} \left[ m+j-1\atop j\right]_{t^2},\\
 \lim_{q \rightarrow 0} \widetilde{B}\big(t^{m+1},j\big) \bigr|_{a=t^{1/2}, \, c=q^{1/2}t^{1/2}}= \left[ m+2j\atop j\right]_{t^2}.
\end{gather*}
%\end{subequations}
\end{cor}

\subsubsection[$(D_n,D_n)$ case]{$\boldsymbol{(D_n,D_n)}$ case}
\begin{prp} If $a=1$, $c=q^{1/2}$, we have
%\begin{subequations}
\begin{gather*}
 B(s,j)= {\big(t/q;t^2\big)_j \big(s^2/t^2,t^2\big)_j\over \big(t^2;t^2\big)_j\big(s^2 q/ t ;t^2\big)_j}{1-s t^{2j-1}\over 1-s/t}q^j ,\\
 \widetilde{B}(s,j)= {\big(q/t;t^2\big)_j \big(s^2 t^{2j},t^2\big)_j\over \big(t^2;t^2\big)_j\big( s^2q t^{2j-3};t^2\big)_j}{1+s/t\over 1+s t^{2j-1}}t^j.
\end{gather*}
%\end{subequations}
\end{prp}
\begin{proof}When $a=1$, $c=q^{1/2}$, we have
\begin{gather*}
B(s,j) =(-1)^j s^{-j} {\big(s^2/t^2; t^2\big)_j \over \big(t^2; t^2\big)_j} {1-s^2 t^{4j-2} \over 1-s^2 t^{-2}}{}_{3}\phi_2 \left[ { -sq, s^2 t^{2j-2}, t^{-2j}\atop -st, s^2q/t } ; t^2, t^2\right]\\
\hphantom{B(s,j)}{} =(-1)^j s^{-j} {\big(s^2/t^2; t^2\big)_j \over \big(t^2; t^2\big)_j}{1-s^2 t^{4j-2} \over 1-s^2 t^{-2}} {\big(t/q,t^2\big)_j \big({-}t^{-2j+3}/s;t^2\big)_j\over
\big({-}st;t^2\big)_j\big(t^{-2j+3}/s^2q;t^2\big)_j}\nonumber \\
\hphantom{B(s,j)}{} ={\big(t/q;t^2\big)_j \big(s^2/t^2,t^2\big)_j\over \big(t^2;t^2\big)_j\big(s^2 q/ t ;t^2\big)_j}{1-s t^{2j-1}\over 1-s/t}q^j,
\end{gather*}
and
\begin{gather*}
\widetilde{B}(s,j) = t^j \big(st^{j-1}\big)^{-j}{\big(s^2t^{2j};t^2\big)_j\over \big(t^2;t^2\big)_j}{ 1+st^{-1} \over 1+st^{2j-1}}
{}_{3}\phi_2 \left[ {- t^{-2j+3}/sq , t^{-2j + 2}/s^2, t^{-2j}\atop -t^{-2j+2}/s, t^{-4j+5}/s^2q} ;t^2, t^2 \right]\\
\hphantom{\widetilde{B}(s,j)}{} =
t^j \big(st^{j-1}\big)^{-j}{\big(s^2t^{2j};t^2\big)_j\over \big(t^2;t^2\big)_j}{ 1+st^{-1} \over 1+st^{2j-1}}
{\big(q/t;t^2\big)_j \big({-}s;t^2\big)_j\over ( -t^{-2j+2}/s;t^2)_j \big(s^2 qt^{2j-3} ;t^2\big)_j} \\
\hphantom{\widetilde{B}(s,j)}{} =
{\big(q/t;t^2\big)_j \big(s^2 t^{2j},t^2\big)_j\over \big(t^2;t^2\big)_j\big( s^2q t^{2j-3};t^2\big)_j}{1+s/t\over 1+s t^{2j-1}}t^j.\tag*{\qed}
\end{gather*}\renewcommand{\qed}{}
\end{proof}

\begin{cor} When $a=1$, $c=q^{1/2}$, $t=q$, we have
%\begin{subequations}
\begin{gather*}
 B(s,j)=\delta_{j,0} ,\qquad \widetilde{B}(s,j)=\delta_{j,0} .
\end{gather*}
%\end{subequations}
\end{cor}

\begin{cor}Let $m\in \mathbb{C}$. We have
%\begin{subequations}
\begin{gather*}
 \lim_{q \rightarrow 0} B\big(t^{m+1},j\big) \big|_{a=1, c=q^{1/2}}=(-1)^j t^{j^2} { [m+2j]_{t} \over [m]_{t}} \left[ m+j-1\atop j\right]_{t^2},\\
 \lim_{q \rightarrow 0} \widetilde{B}\big(t^{m+1},j\big) \big|_{a=1, c=q^{1/2}}= t^j{1+t^m\over 1+t^{m+2j}}\left[ m+2j\atop j\right]_{t^2}.
\end{gather*}
%\end{subequations}
\end{cor}

\subsection[Explicit formulas for $P^{(C_n,C_n)}_{(1^{r})}(x\,|\,t;q,t)$ and $P^{(D_n,D_n)}_{(1^{r})}(x\,|\,q,t)$]{Explicit formulas for $\boldsymbol{P^{(C_n,C_n)}_{(1^{r})}(x\,|\,t;q,t)}$ and $\boldsymbol{P^{(D_n,D_n)}_{(1^{r})}(x\,|\,q,t)}$}
Using the formulas obtained in the previous subsection, we give some explicit transition formulas for the polynomials $P^{(C_n,C_n)}_{(1^{r})}(x\,|\,t;q,t)$ and $P^{(D_n,D_n)}_{(1^{r})}(x\,|\,q,t)=P^{(C_n,C_n)}_{(1^{r})}(x\,|\,1;q,t)$.
\begin{thm}We have
%\begin{subequations}
\begin{gather*}
 P^{(C_n,C_n)}_{(1^r)}(x\,|\,t;q,t) =\sum_{j=0}^{\lfloor {r \over 2} \rfloor}{\big(1/qt;t^2\big)_j \big(t^{2n-2r},t^2\big)_j\over \big(t^2;t^2\big)_j\big(q t^{2n-2r+3} ;t^2\big)_j}{1-t^{2n-2r+4j} \over 1-t^{2n-2r}}(q t)^jE_{r-2j}(x),\\
E_{r}(x)=\sum_{j=0}^{\lfloor {r \over 2} \rfloor}{\big(qt;t^2\big)_j \big(t^{2n-2r+2j+2},t^2\big)_j\over \big(t^2;t^2\big)_j\big( q t^{2n-2r+2j+1};t^2\big)_j}P^{(C_n,C_n)}_{(1^{r-2j})}(x\,|\,t;q,t),\\
P^{(D_n,D_n)}_{(1^r)}(x\,|\,q,t) =\sum_{j=0}^{\lfloor {r \over 2} \rfloor}{\big(t/q;t^2\big)_j \big(t^{2n-2r},t^2\big)_j\over \big(t^2;t^2\big)_j\big(q t^{2n-2r+1} ;t^2\big)_j}{1-t^{n-r+2j} \over 1-t^{n-r}}q^jE_{r-2j}(x),\\
E_{r}(x)=\sum_{j=0}^{\lfloor {r \over 2} \rfloor}{\big(q/t;t^2\big)_j \big(t^{2n-2r+2j+2},t^2\big)_j\over \big(t^2;t^2\big)_j\big( q t^{2n-2r+2j-1};t^2\big)_j}{1+t^{n-r}\over 1+t^{n-r+2j}}t^jP^{(D_n,D_n)}_{(1^{r-2j})}(x\,|\,q,t).
\end{gather*}
%\end{subequations}
\end{thm}

\begin{cor}\label{Schur}Setting $t=q$, we have the formula for the Schur polynomials $s^{(C_n)}_{(1^r)}(x)=P^{(C_n,C_n)}_{(1^r)}(x\,|\,q;q,q)$ and $s^{(D_n)}_{(1^r)}(x)=P^{(D_n,D_n)}_{(1^r)}(x\,|\,q,q)$:
%\begin{subequations}
\begin{gather*}
 s^{(C_n)}_{(1^r)}(x) =E_{r}(x)-E_{r-2}(x),\qquad
E_{r}(x)=\sum_{j=0}^{\lfloor {r \over 2} \rfloor} s^{(C_n)}_{(1^{r-2k})}(x),\qquad
s^{(D_n)}_{(1^r)}(x)=E_{r}(x).
\end{gather*}
%\end{subequations}
Hence, from Lemma~{\rm \ref{Lem-Em}}, we have
%\begin{subequations}
\begin{gather*}
s^{(C_n)}_{(1^r)}(x) =\sum_{j=0}^{\lfloor{r \over 2}\rfloor}\left( \binom{ n-r+2j }{ j} - \binom{ n-r+2j }{ j-1} \right)m_{(1^{r-2j})}(x)\\
\hphantom{s^{(C_n)}_{(1^r)}(x)}{} = \sum_{j=0}^{\lfloor{r \over 2}\rfloor}{n-r+1\over n-r+j+1}\binom{ n-r+2j }{ j} m_{(1^{r-2j})}(x),\\
s^{(D_n)}_{(1^r)}(x)=\sum_{j=0}^{\lfloor{r \over 2}\rfloor} \binom{ n-r+2j }{ j} m_{(1^{r-2j})}(x).
\end{gather*}
%\end{subequations}
\end{cor}

\subsection[Hall--Littlewood polynomials $P^{(C_n,C_n)}_{(1^{r})}(x\,|\,t;0,t)$, $P^{(D_n,D_n)}_{(1^{r})}(x\,|\,0,t)$ and Kostka polynomials]{Hall--Littlewood polynomials $\boldsymbol{P^{(C_n,C_n)}_{(1^{r})}(x\,|\,t;0,t)}$, $\boldsymbol{P^{(D_n,D_n)}_{(1^{r})}(x\,|\,0,t)}$\\ and Kostka polynomials}\label{Kostka}

Using the transition formulas we have established, we can study the Kostka polynomials associated with one column diagrams for types $C_n$ and $D_n$. Setting $b=t$, $q=0$ for $(C_n,C_n)$ (or $b=1$, $q=0$ for $(D_n,D_n)$) in
$P^{(C_n,C_n)}_{(1^r)}(x\,|\,b;q,t)$, we have the type $C_n$ (or type~$D_n$) Hall--Littlewood polynomials with one column diagrams.

\begin{thm}We have
%\begin{subequations}
\begin{gather*}
P^{(C_n,C_n)}_{(1^r)}(x\,|\,t;0,t)= \sum_{j=0}^{\lfloor {r \over 2} \rfloor}(-1)^j t^{j(j-1)}{ [n-r+2j]_{t^2} \over [n-r]_{t^2}} \left[ n-r+j-1\atop j\right]_{t^2}E_{r-2j}(x),\\
E_{r}(x)=\sum_{j=0}^{\lfloor {r \over 2} \rfloor} \left[ n-r+2j\atop j\right]_{t^2}P^{(C_n,C_n)}_{(1^{r-2j})}(x\,|\,t;0,t),\\
P^{(D_n,D_n)}_{(1^r)}(x\,|\,0,t)=\sum_{j=0}^{\lfloor {r \over 2} \rfloor}(-1)^j t^{j^2}{ [n-r+2j]_{t} \over [n-r]_{t}} \left[ n-r+j-1\atop j\right]_{t^2}E_{r-2j}(x),\\
E_{r}(x)=\sum_{j=0}^{\lfloor {r \over 2} \rfloor}t^j {1+t^{n-r}\over 1+t^{n-r+2j}}\left[ n-r+2j\atop j\right]_{t^2}P^{(D_n,D_n)}_{(1^{r-2j})}(x\,|\,0,t).
\end{gather*}
%\end{subequations}
\end{thm}

Then, applying the formulas for the Schur polynomials in Corollary~\ref{Schur}, we can calculate the Kostka polynomials (i.e., the transition coefficients from the Schur polynomials to the Hall--Littlewood polynomials) of types $C_n$ and $D_n$ associated with one column diagrams as follows.

\begin{thm}We have
\begin{subequations}
\begin{gather}
s^{(C_n)}_{(1^r)}(x) =\sum_{j=0}^{\lfloor{r \over 2}\rfloor} t^{2j}{[n-r+1]_{t^2}\over [n-r+j+1]_{t^2}}{n-r+2j \atopwithdelims[] j}_{t^2} P^{(C_n,C_n)}_{(1^{r-2j})}(x\,|\,t;0,t)\nonumber\\
\hphantom{s^{(C_n)}_{(1^r)}(x)}{} =\sum_{j=0}^{\lfloor{r \over 2}\rfloor}\left( {n-r+2j \atopwithdelims[] j}_{t^2} - {n-r+2j \atopwithdelims[] j-1}_{t^2} \right)
P^{(C_n,C_n)}_{(1^{r-2j})}(x\,|\,t;0,t),\label{Kostka-C}\\
s^{(D_n)}_{(1^r)}(x)=\sum_{j=0}^{\lfloor {r \over 2} \rfloor} t^j {1+t^{n-r}\over 1+t^{n-r+2j}}\left[ n-r+2j\atop j\right]_{t^2} P^{(D_n,D_n)}_{(1^{r-2j})}(x\,|\,0,t) \nonumber\\
\hphantom{s^{(D_n)}_{(1^r)}(x)}{} = \sum_{j=0}^{\lfloor {r \over 2} \rfloor}\!\left(\!t^{n-r+j}\left[ n-r+2j-1\atop j-1\right]_{t^2}+t^j\left[ n-r+2j-1\atop j\right]_{t^2}\!\right)P^{(D_n,D_n)}_{(1^{r-2j})}(x\,|\,0,t).\!\!\!\!\label{Kostka-D}
\end{gather}
\end{subequations}
Hence we have Theorem~{\rm \ref{ProofMAIN}}.
\end{thm}

\begin{rmk}The expansion coefficient of (\ref{Kostka-C}) (times $t^{-2j}$) is identified with the $q$-ballot (when $m=0$, $q$-Catalan) number \cite{A, FH}
\begin{gather*}
q^{-j}\left( {m+2j \atopwithdelims[] j}_{q} - {m+2j \atopwithdelims[] j-1}_{q} \right)={[m+1]_{q}\over [m+j+1]_{q}} {m+2j \atopwithdelims[] j}_{q},
\end{gather*}
by the replacement $m\rightarrow n-r$, $q\rightarrow t^2$. The case $m=0$ gives us the $q$-Catalan number. It is known that the $q$-Catalan or $q$-ballot number is a polynomial in~$q$ with positive integral coefficients (see~\cite{A,FH}).

The expansion coefficient of (\ref{Kostka-D}) is identified with the following version of the $q$-binomial number
\begin{gather*}
q^j {1+q^{m-2j}\over 1+q^{m}} \left[ m\atop j\right]_{q^2}=q^{m-j}\left[ m-1\atop j-1\right]_{q^2}+q^j\left[ m-1\atop j\right]_{q^2},
\end{gather*}
by the replacement $m\rightarrow n-r+2j$, $q\rightarrow t$. Note that this is also a polynomial in~$q$ with positive integral coefficients.
\end{rmk}

\section[Some conjectures about Macdonald polynomials of type $C_n$]{Some conjectures about Macdonald polynomials of type $\boldsymbol{C_n}$}\label{CON}

\subsection[Asymptotically free eigenfunctions for the Macdonald operator of type $A_{n-1}$]{Asymptotically free eigenfunctions for the Macdonald operator\\ of type $\boldsymbol{A_{n-1}}$}
First we recall some facts about the asymptotically free eigenfunctions for the case $A_{n-1}$. Let $n\in\mathbb{Z}_{>0}$, and $q,t\in\mathbb{C}$ be generic parameters. Let $x=(x_1,\ldots,n_n)$ be a sequence of independent indeterminates. Macdonald's difference operator of type $A_{n-1}$ is defined by
\begin{gather*}
D^{(A_{n-1})}=\sum_{i=1}^n \prod_{j\neq i}{t x_i-x_j\over x_i-x_j} T_{q,x_i}.
\end{gather*}
For a partition $\lambda$ with $\ell(\lambda)\leq n$, the Macdonald symmetric polynomial $P_\lambda(x;q,t)\in \mathbb{C}[x_1,\ldots$, $x_n]^{S_n}$ exists uniquely characterized by the conditions:
\begin{gather*}
P_\lambda=m_\lambda+\sum_{\mu<\lambda}u_{\lambda\mu} m_\mu, \qquad
D^{(A_{n-1})} P_\lambda=\sum_{i=1}^n q^{\lambda_i}t^{n-i} \cdot P_\lambda.
\end{gather*}

Let $s_1,s_2,\dots,s_n\in \mathbb{C}$ be complex variables. Let $\mathsf{M}^{(n)}$ be the set of strict upper triangular matrices with entries in $\mathbb{Z}_{\geq 0}$, namely for $\theta^{(n)}=\big(\theta^{(n)}_{ij}\big)_{i,j \in \mathbb{Z}_{\geq 0}}\in\mathsf{M}^{(n)}$ $i\geq j$ implies $\theta^{(n)}_{ij}=0$.

\begin{dfn}For $n\geq 1$, define recursively the rational functions $c_n(\theta^{(n)};s_1,\dots,s_n;q,t)\in {\bf \mathbb{Q}}(q,t,s_1,\dots,s_n)$ by $c_1(-;s_1;q,t)=1$, and
\begin{gather*}
c_n\big(\theta^{(n)};s_1,\dots,s_n;q,t\big) =c_{n-1}\big(\theta^{(n-1)}; q^{-\theta_{1,n}}s_1,\dots,q^{-\theta_{n-1,n}}s_{n-1};q,t\big) \nonumber\\
\hphantom{c_n\big(\theta^{(n)};s_1,\dots,s_n;q,t\big) =}{} \times
\prod_{1\leq i\leq j\leq n-1} {(t s_{j+1}/s_i;q)_{\theta_{i,n}}\over (q s_{j+1}/s_i;q)_{\theta_{i,n}}}
{\big(q^{-\theta_{j,n}}q s_j/ts_i;q\big)_{\theta_{i,n}}\over \big(q^{-\theta_{j,n}} s_j/s_i;q\big)_{\theta_{i,n}}}.
\end{gather*}
\end{dfn}
\begin{dfn}Set
\begin{gather*}
 \varphi^{(A_{n-1})}(s\,|\,x)= \sum_{\theta^{(n)}\in \mathsf{M}^{(n)}} c_n\big(\theta^{(n)};s_1,\dots,s_n;q,t\big) \prod_{1\leq i<j\leq n} \left({x_j\over x_i}\right)^{\theta_{ij}}.
\end{gather*}
\end{dfn}

\begin{thm}[\cite{BFS, NS}]Write $s_i=t^{n-1} q^{\lambda_i}$, $1\leq i\leq n$ for simplicity. We have
\begin{gather*}
D^{(A_{n-1})}x^\lambda\varphi^{(A_{n-1})}(s\,|\,x)=(s_1+\cdots+s_n) x^\lambda \varphi^{(A_{n-1})}(s\,|\,x).
\end{gather*}
When $\lambda$ is a partition with $\ell(\lambda)\leq n$, we have
\begin{gather*}
x^\lambda\varphi^{(A_{n-1})}(s\,|\,x)= P_\lambda(x).
\end{gather*}
\end{thm}

\begin{rmk}[branching formulas] We have the decomposition of the series $\varphi^{(A_{n-1})}$ in terms of the $\varphi^{(A_{n-2})}$ series as
\begin{gather*}
 \varphi^{(A_{n-1})}(s_1\ldots,s_n\,|\,x_1\ldots,x_n) \\
 \qquad{} =\sum_{\theta_{1n},\ldots,\theta_{n-1,n}\geq 0} \varphi^{(A_{n-2})}\big(q^{-\theta_{1n}}s_1,\ldots,q^{-\theta_{n-1,n}}s_{n-1}\,|\,
x_1,\ldots,x_{n-1}\big) \\
 \qquad\quad {}\times \prod_{1\leq i\leq j\leq n-1}{(t s_{j+1}/s_i;q)_{\theta_{i,n}}\over (q s_{j+1}/s_i;q)_{\theta_{i,n}}}{\big(q^{-\theta_{j,n}}q s_j/ts_i;q\big)_{\theta_{i,n}}\over \big(q^{-\theta_{j,n}} s_j/s_i;q\big)_{\theta_{i,n}}}\cdot
\prod_{i=1}^{n-1}\left({x_n\over x_i}\right)^{\theta_{in}}.
\end{gather*}
\end{rmk}

\subsection[Asymptotically free eigenfunction of type $C_n$]{Asymptotically free eigenfunction of type $\boldsymbol{C_n}$}
Let $n\in \mathbb{Z}_{>0}$. Let $x=(x_1,\ldots,x_n)$ and $(s_1,\ldots,s_n)$ be a pair of variables. Let ${\mathcal D}_x ^{(C_n)}={\mathcal D}_x\big({-}t^{1/2},t^{1/2}-q^{1/2}t^{1/2},q^{1/2}t^{1/2} \,|\,q,t\big)$ be the $BC_n$ Koornwinder operator degenerated to the $C_n$ case.

\begin{dfn} Set
\begin{gather*}
s_i=t^{n-i+1} q^{\lambda_i}, \qquad 1\leq i\leq n.
\end{gather*}
We define the asymptotically free eigenfunction $x^\lambda \varphi^{(C_n)}(s\,|\,x)$ of type $C_n$ by
\begin{gather*}
\varphi^{(C_n)}(s\,|\,x)=\varphi^{(C_n)}(s_1,\ldots,s_n\,|\,x_1,\ldots,x_n)\\
\hphantom{\varphi^{(C_n)}(s\,|\,x)}{}= \sum_{k_1,\ldots,k_n\geq 0} c_{k_1,\ldots,k_n}(s_1,\ldots,s_n;q,t) \left( x_2\over x_1\right)^{k_1}\cdots \left( x_n\over x_{n-1}\right)^{k_{n-1}} \left( 1\over x_{n}^2\right)^{k_{n}},\\
{\mathcal D}_x ^{(C_n)} x^\lambda \varphi^{(C_n)}(s\,|\,x)=\varepsilon^{(C_n)}(s) x^\lambda \varphi^{(C_n)}(s\,|\,x),\\
\varepsilon^{(C_n)}(s)=\sum_{i=1}^n \big(s_i+s_i^{-1}-t^i-t^{-i}\big).
\end{gather*}
\end{dfn}

\subsection[$C_2$ case]{$\boldsymbol{C_2}$ case}
\begin{dfn}
Let $x_1$, $x_2$, $s_1$, $s_2$ be variables. Set
\begin{gather*}
\psi^{(C_2)}(s_1,s_2\,|\,x_1,x_2)\\
=\sum_{\theta_{12},\mu_{12},\rho_1,\rho_2\geq 0}c^{(C_2)}(\theta_{12},\mu_{12},\rho_1,\rho_2;s_1,s_2;q,t)
\left({x_2\over x_1}\right)^{\theta_{12}}
\left({1\over x_1x_2}\right)^{\mu_{12}}
\left({1\over x_1^2}\right)^{\rho_{1}}
\left({1\over x_2^2}\right)^{\rho_{2}},
\end{gather*}
where
\begin{gather*}
 c^{(C_2)}(\theta_{12},\mu_{12},\rho_1,\rho_2;s_1,s_2;q,t) ={(t)_{\theta_{12}}\over (q)_{\theta_{12}}}{(t s_2/s_1)_{\theta_{12}}\over (q s_2/s_1)_{\theta_{12}}}(q/t)^{\theta_{12}}
{(t)_{\mu_{12}}\over (q)_{\mu_{12}}}{(t /s_1s_2)_{\mu_{12}}\over (q /s_1s_2)_{\mu_{12}}}(q/t)^{\mu_{12}} \\
\qquad{}\times {(t/s_2)_{\mu_{12}}\over (q/s_2)_{\mu_{12}}}{\big(q^{-\theta_{12}}q/ts_2\big)_{\mu_{12}}\over \big(q^{-\theta_{12}}/s_2\big)_{\mu_{12}}}{(t/s_1)_{\mu_{12}}\over (q/s_1)_{\mu_{12}}}
{\big(q^{\theta_{12}}q/s_1\big)_{\mu_{12}}\over \big(q^{\theta_{12}}t/s_1\big)_{\mu_{12}}} \\
\qquad{}\times {(t)_{\rho_{1}}\over (q)_{\rho_{1}}}{\big(q^{\theta_{12}+\mu_{12}}t^2 /s_1\big)_{\rho_{1}}\over \big(q^{\theta_{12}+\mu_{12}}qt /s_1\big)_{\rho_{1}}}(q/t)^{\rho_{1}}{(t)_{\rho_{2}}\over (q)_{\rho_{2}}}
{\big(q^{-\theta_{12}+\mu_{12}}t /s_2\big)_{\rho_{2}}\over \big(q^{-\theta_{12}+\mu_{12}}q /s_2\big)_{\rho_{2}}}(q/t)^{\rho_{2}}.
\end{gather*}
\end{dfn}

\begin{con}We have $\psi^{(C_2)}(s_1,s_2\,|\,x_1,x_2)=\varphi^{(C_2)}(s_1,s_2\,|\,x_1,x_2)$. Namely, setting $s_1=t^2 q^{\lambda_1}$, $s_2=t q^{\lambda_2}$, $x^\lambda=x_1^{\lambda_1}x_2^{\lambda_2}$, we have
\begin{gather*}
 {\mathcal D}_x^{(C_2)} x^\lambda \psi^{(C_2)}(s\,|\,x)= \varepsilon^{(C_2)}(s) x^\lambda \psi^{(C_2)}(s\,|\,x), \\
\varepsilon^{(C_2)}(s)=s_1+s_2+s_2^{-1}+s_1^{-1}-t^2-t-t^{-1}-t^{-2}.
\end{gather*}
When $\lambda=(\lambda_1,\lambda_2)$ is a partition, we have
\begin{gather*}
x^\lambda \psi^{(C_2)}(s\,|\,x)= P^{(C_2)}_{\lambda}(x\,|\,t;q,t).
\end{gather*}
\end{con}

\subsection[$C_3$ case with rectangular diagrams]{$\boldsymbol{C_3}$ case with rectangular diagrams}
We can study the decomposition of the $C_3$ Macdonald polynomials $P^{(C_3)}_{\lambda}(x\,|\,t;q,t)$ in terms of the $C_2$ Macdonald polynomials. It seems that such a decomposition becomes rather simple when we consider the case of a rectangular diagram consisting of three equal rows $\lambda=(\lambda_3,\lambda_3,\lambda_3)$,

\begin{dfn}Let $\lambda_3\in \mathbb{C}$ and set
\begin{gather}
s_1=t^2s_3,\qquad s_2=t s_3,\qquad s_3=t q^{\lambda_3}. \label{rect}
\end{gather}
Define
\begin{gather*}
 \psi^{(C_3),{\rm rect}}(s_3\,|\,x_1,x_2,x_3) =\sum_{\mu_{13},\rho_1\geq 0}{(t)_{\mu_{13}}\over (q)_{\mu_{13}}}{\big(1 /s_3^2\big)_{\mu_{13}}\over \big(q /ts_3^2\big)_{\mu_{13}}}(q/t)^{\mu_{13}}
{(t/s_3)_{\mu_{13}}\over (q/s_3)_{\mu_{13}}}{(q/ts_3)_{\mu_{13}}\over (1/s_3)_{\mu_{13}}} \\
\qquad{}\times {(t)_{\rho_{1}}\over (q)_{\rho_{1}}}{\big(q^{\mu_{13}}t /s_3\big)_{\rho_{1}}\over \big(q^{\mu_{13}}q /s_3\big)_{\rho_{1}}}(q/t)^{\rho_{1}}
\left( 1\over x_1x_3\right)^{\mu_{13}} \left( 1\over x_1^2\right)^{\rho_{1}}
\varphi^{(C_2)}\big(t s_3,q^{-\mu_{13}}s_3 \,|\,x_2,x_3\big).
\end{gather*}
\end{dfn}

\begin{con}We have $\psi^{(C_3),{\rm rect}}(s_3\,|\,x_1,x_2,x_3)=\varphi^{(C_3)}\big(t^2s_3,t s_3,s_3\,|\,x_1,x_2,x_3\big)$. Namely, setting $x^\lambda=(x_1x_2x_3)^{\lambda_3}$, we have
\begin{gather*}
 {\mathcal D}_x ^{(C_3)} x^\lambda \varphi^{(C_3),{\rm rect}}(s_3\,|\,x_1,x_2,x_3)= \varepsilon^{(C_3)}(s) x^\lambda \varphi^{(C_3),{\rm rect}}(s_3\,|\,x_1,x_2,x_3), \\
\varepsilon^{(C_3)}(s)=s_1+s_2+s_3+s_3^{-1}+s_2^{-1}+s_1^{-1}-t^3-t^2-t-t^{-1}-t^{-2}-t^{-3},
\end{gather*}
where $s_1$, $s_2$, $s_3$ are as given in \eqref{rect}. When $\lambda_3$ is a nonnegative integer, we have
\begin{gather*}
x^\lambda \varphi^{(C_3),{\rm rect}}(s_3\,|\,x_1,x_2,x_3)= P^{(C_3)}_{(\lambda_3,\lambda_3,\lambda_3)}(x\,|\,t;q,t).
\end{gather*}
\end{con}

\subsection[Folding of $A_{2n-1}$ eigenfunctions and decomposition with respect to $C_n$ eigenfunctions]{Folding of $\boldsymbol{A_{2n-1}}$ eigenfunctions and decomposition\\ with respect to $\boldsymbol{C_n}$ eigenfunctions}

\begin{dfn} Let $n\in \mathbb{Z}_{>0}$. Let $x=(x_1,\ldots,x_{n})$ and
\begin{gather*}
s=(s_1,\ldots,s_n),\qquad s_i=t^{n-i+1} q^{\lambda_i}, \qquad 1\leq i\leq n,
\end{gather*}
be a pair of variables. Define the folded series $\widetilde{\varphi}^{(A_{2n-1})}(s\,|\,x)=\widetilde{\varphi}^{(A_{2n-1})}(s_1,\ldots,s_n\,|\,x_1,\ldots,x_{n})$ by
\begin{gather*}
\widetilde{\varphi}^{(A_{2n-1})}(s\,|\,x) = \varphi^{(A_{2n-1})} \big(t^{n-1}s_1,\ldots,t^{n-1}s_n,t^{n-1},t^{n-2},\ldots,t,1\,|\,x_1,\ldots,x_n,x_n^{-1},\ldots,x_1^{-1}\big).
\end{gather*}
\end{dfn}

\begin{prp}When $n=1$, we have
\begin{gather*}
\widetilde{\varphi}^{(A_{1})}(s\,|\,x)=\varphi^{(A_{1})}\big(s_1,1\,|\,x_1,x_1^{-1}\big) =\sum_{\theta \geq 0} {(t;q)_\theta \over (q;q)_\theta}{(ts_1;q)_\theta \over (qs_1;q)_\theta}(q /t)^\theta x_1^{2\theta}.
\end{gather*}
Hence we have $\widetilde{\varphi}^{(A_{1})}(s\,|\,x)=\varphi^{(C_1)}(s_1\,|\,x_1)$.
\end{prp}

We calculated the decomposition of the folded eigenfunctions $\widetilde{\varphi}^{(A_{2n-1})}(s\,|\,x)$ with respect to the $C_n$ series $\varphi^{(C_n)}(s\,|\,x)$ for the cases $n=2$ and $3$ using Mathematica.

\begin{con} We have
 \begin{gather*}
\widetilde{\varphi}^{(A_{3})}(s_1,s_2\,|\,x_1,x_2)= \sum_{\mu_{12}\geq 0} e_2(s_1,s_2;\mu_{12}) \left(1\over x_1x_2\right)^{\mu_{12}} \varphi^{(C_2)}\big(q^{-\mu_{12}}s_1,q^{-\mu_{12}}s_2\,|\,x_1,x_2\big),
 \end{gather*}
 where
 \begin{gather*}
 e_2(s_1,s_2;\mu_{12}) = {(t/s_1)_{\mu_{12} }\over (q/s_1)_{\mu_{12} }} {(t/s_2)_{\mu_{12} }\over (q/s_2)_{\mu_{12} }} {(t)_{\mu_{12} }\over (q)_{\mu_{12} }} {\big(q^{\mu_{12}} q/ts_1s_2\big)_{\mu_{12} }\over
 \big(q^{\mu_{12}} /s_1s_2\big)_{\mu_{12} }}(q/t)^{\mu_{12}}.
 \end{gather*}
 \end{con}

 \begin{con}
 We have
 \begin{gather*}
 \widetilde{\varphi}^{(A_{5})}(s_1,s_2,s_3\,|\,x_1,x_2,x_3)=
 \sum_{\mu_{12},\mu_{13},\mu_{23}\geq 0}
 e_3(s_1,s_2,s_3;\mu_{12},\mu_{13},\mu_{23})\left(1\over x_1x_2\right)^{\mu_{12}} \\
 \qquad{} \times \left(1\over x_1x_3\right)^{\mu_{13}} \left(1\over x_2x_3\right)^{\mu_{23}}
 \varphi^{(C_3)}\big(q^{-\mu_{12}-\mu_{13}}s_1,q^{-\mu_{12}-\mu_{23}}s_2,
 q^{-\mu_{13}-\mu_{23}} s_3\,|\,x_1,x_2,x_3\big),
 \end{gather*}
 where
 \begin{gather*}
e_3(s_1,s_2,s_3;\mu_{12},\mu_{13},\mu_{23}) =
 {(t/s_1)_{\mu_{12} +\mu_{13 }}\over (q/s_1)_{\mu_{12}+\mu_{13 } }}
 {(t/s_2)_{\mu_{12}+\mu_{23 } }\over (q/s_2)_{\mu_{12}+\mu_{23 } }}
 {(t/s_3)_{\mu_{13}+\mu_{23 } }\over (q/s_3)_{\mu_{13}+\mu_{23 } }}\\
\qquad{}\times
 {(t)_{\mu_{12} }\over (q)_{\mu_{12} }}
 {\big(q^{\mu_{12}+\mu_{13}+\mu_{23}} q/ts_1s_2\big)_{\mu_{12} }\over
 \big(q^{\mu_{12}+\mu_{13}+\mu_{23}}/s_1s_2\big)_{\mu_{12} }}(q/t)^{\mu_{12}} \\
\qquad{}\times
 {(t s_3/s_1)_{\mu_{12} } \over (q s_3/s_1)_{\mu_{12} } }
 {\big(q^{-\mu_{23}}q s_3/ts_1\big)_{\mu_{12} } \over \big(q^{-\mu_{23}} s_3/s_1\big)_{\mu_{12} } }
 {(t s_3/s_2)_{\mu_{12} } \over (q s_3/s_2)_{\mu_{12} } }
 {\big(q^{-\mu_{13}}q s_3/ts_2\big)_{\mu_{12} } \over \big(q^{-\mu_{13}} s_3/s_2\big)_{\mu_{12} } } \\
\qquad{}\times
 {(t)_{\mu_{13} }\over (q)_{\mu_{13} }}
 {\big(q^{\mu_{12}+\mu_{13}+\mu_{23}} q/ts_1s_3\big)_{\mu_{13} }\over
 \big(q^{\mu_{12}+\mu_{13}+\mu_{23}}/s_1s_3\big)_{\mu_{13} }}(q/t)^{\mu_{13}}
 {(t s_2/s_1)_{\mu_{13} } \over (q s_2/s_1)_{\mu_{13} } }
 {\big(q^{-\mu_{23}}q s_2/ts_1\big)_{\mu_{13} } \over \big(q^{-\mu_{23}} s_2/s_1\big)_{\mu_{13} } } \\
\qquad{}\times
 {(t)_{\mu_{23} }\over (q)_{\mu_{23} }}
 {\big(q^{\mu_{12}+\mu_{13}+\mu_{23}} q/ts_2s_3\big)_{\mu_{23} }\over
 \big(q^{\mu_{12}+\mu_{13}+\mu_{23}}/s_2s_3\big)_{\mu_{23} }}(q/t)^{\mu_{23}}.
 \end{gather*}
\end{con}

\subsection*{Acknowledgements}
Research of A.H.\ is supported by JSPS KAKENHI (Grant Number 16K05186). Research of J.S.\ is supported by JSPS KAKENHI (Grant Numbers 15K04808 and 16K05186). The authors thank M.~Noumi, B.~Feigin and H.~Awata for stimulating discussion. They thank the anonymous referees for their various constructive comments.

\pdfbookmark[1]{References}{ref}
\LastPageEnding


\begin{thebibliography}{99}
\footnotesize\itemsep=0pt

\bibitem{A}
Allen E.A., Combinatorial interpretations of generalizations of {C}atalan
 numbers and ballot numbers, Ph.D.~Thesis, Carnegie Mellon University, 2014.

\bibitem{AW}
Askey R., Wilson J., Some basic hypergeometric orthogonal polynomials that
 generalize {J}acobi polynomials, \href{https://doi.org/10.1090/memo/0319}{\textit{Mem. Amer. Math. Soc.}} \textbf{54}
 (1985), iv+55~pages.

\bibitem{BFS}
Braverman A., Finkelberg M., Shiraishi J., Macdonald polynomials, {L}aumon
 spaces and perverse coherent sheaves, in Perspectives in Representation
 Theory, \href{https://doi.org/10.1090/conm/610/12130}{\textit{Contemp. Math.}}, Vol.~610, Amer. Math. Soc., Providence, RI,
 2014, 23--41, \href{https://arxiv.org/abs/1206.3131}{arXiv:1206.3131}.

\bibitem{B}
Bressoud D.M., A matrix inverse, \href{https://doi.org/10.2307/2044991}{\textit{Proc. Amer. Math. Soc.}} \textbf{88}
 (1983), 446--448.

\bibitem{FHNSS}
Feigin B., Hoshino A., Noumi M., Shibahara J., Shiraishi J., Tableau formulas
 for one-row {M}acdonald polynomials of types {$C_n$} and {$D_n$},
 \href{https://doi.org/10.3842/SIGMA.2015.100}{\textit{SIGMA}} \textbf{11} (2015), 100, 21~pages, \href{https://arxiv.org/abs/1412.8001}{arXiv:1412.8001}.

\bibitem{FH}
F\"{u}rlinger J., Hofbauer J., {$q$}-{C}atalan numbers, \href{https://doi.org/10.1016/0097-3165(85)90089-5}{\textit{J.~Combin.
 Theory Ser.~A}} \textbf{40} (1985), 248--264.

\bibitem{GR}
Gasper G., Rahman M., Basic hypergeometric series, \href{https://doi.org/10.1017/CBO9780511526251}{\textit{Encyclopedia of
 Mathematics and its Applications}}, Vol.~96, 2nd~ed., Cambridge University
 Press, Cambridge, 2004.

\bibitem{HNS}
Hoshino A., Noumi M., Shiraishi J., Some transformation formulas associated
 with {A}skey--{W}ilson polynomials and {L}assalle's formulas for
 {M}acdonald--{K}oornwinder polynomials, \textit{Mosc. Math.~J.} \textbf{15}
 (2015), 293--318, \href{https://arxiv.org/abs/1406.1628}{arXiv:1406.1628}.

\bibitem{KNS}
Komori Y., Noumi M., Shiraishi J., Kernel functions for difference operators of
 {R}uijsenaars type and their applications, \href{https://doi.org/10.3842/SIGMA.2009.054}{\textit{SIGMA}} \textbf{5} (2009),
 054, 40~pages, \href{https://arxiv.org/abs/0812.0279}{arXiv:0812.0279}.

\bibitem{K}
Koornwinder T.H., Askey--{W}ilson polynomials for root systems of type {$BC$},
 in Hypergeometric Functions on Domains of Positivity, {J}ack Polynomials, and
 Applications ({T}ampa, {FL}, 1991), \href{https://doi.org/10.1090/conm/138/1199128}{\textit{Contemp. Math.}}, Vol.~138, Amer.
 Math. Soc., Providence, RI, 1992, 189--204.

\bibitem{Kr}
Krattenthaler C., A new matrix inverse, \href{https://doi.org/10.1090/S0002-9939-96-03042-0}{\textit{Proc. Amer. Math. Soc.}}
 \textbf{124} (1996), 47--59.

\bibitem{L}
Lassalle M., Some conjectures for {M}acdonald polynomials of type {$B$}, {$C$},
 {$D$}, \textit{S\'{e}m. Lothar. Combin.} \textbf{52} (2004), Art.~B52h,
 24~pages, \href{https://arxiv.org/abs/math.CO/0503149}{math.CO/0503149}.

\bibitem{M1}
Macdonald I.G., Symmetric functions and {H}all polynomials, 2nd~ed., \textit{Oxford
 Mathematical Monographs}, The Clarendon Press, Oxford University Press, New
 York, 1995.

\bibitem{M2}
Macdonald I.G., Orthogonal polynomials associated with root systems,
 \textit{S\'{e}m. Lothar. Combin.} \textbf{45} (2000), Art.~B45a, 40~pages,
 \href{https://arxiv.org/abs/math.QA/0011046}{math.QA/0011046}.

\bibitem{Mi}
Mimachi K., A duality of {M}ac{D}onald--{K}oornwinder polynomials and its
 application to integral representations, \href{https://doi.org/10.1215/S0012-7094-01-10723-0}{\textit{Duke Math.~J.}} \textbf{107}
 (2001), 265--281.

\bibitem{NS}
Noumi M., Shiraishi J., A direct approach to the bispectral problem for the
 {R}uijsenaars--{M}acdonald $q$-difference operators, \href{https://arxiv.org/abs/1206.5364}{arXiv:1206.5364}.

\bibitem{O}
Okounkov A., {${\rm BC}$}-type interpolation {M}acdonald polynomials and
 binomial formula for {K}oornwinder polynomials, \href{https://doi.org/10.1007/BF01236432}{\textit{Transform. Groups}}
 \textbf{3} (1998), 181--207, \href{https://arxiv.org/abs/q-alg/9611011}{q-alg/9611011}.

\bibitem{R1}
Rains E.M., {$BC_n$}-symmetric {A}belian functions, \href{https://doi.org/10.1215/S0012-7094-06-13513-5}{\textit{Duke Math.~J.}}
 \textbf{135} (2006), 99--180, \href{https://arxiv.org/abs/math.CO/0402113}{math.CO/0402113}.

\bibitem{R2}
Rains E.M., Transformations of elliptic hypergeometric integrals, \href{https://doi.org/10.4007/annals.2010.171.169}{\textit{Ann.
 of Math.}} \textbf{171} (2010), 169--243, \href{https://arxiv.org/abs/math.QA/0309252}{math.QA/0309252}.

\bibitem{RW}
Rains E.M., Warnaar S.O., Bounded {L}ittlewood identities, \textit{Mem. Amer.
 Math. Soc.} {t}o appear, \href{https://arxiv.org/abs/1506.02755}{arXiv:1506.02755}.

\bibitem{S}
Shapiro L.W., A {C}atalan triangle, \href{https://doi.org/10.1016/0012-365X(76)90009-1}{\textit{Discrete Math.}} \textbf{14} (1976),
 83--90.

\bibitem{St}
Stokman J.V., Macdonald--{K}oornwinder polynomials, \href{https://arxiv.org/abs/1111.6112}{arXiv:1111.6112}.

\end{thebibliography}
\end{document}